\documentclass[graybox]{svmult}

\usepackage{mathptmx}       
\usepackage{helvet}         
\usepackage{courier}        
\usepackage{type1cm}        
\usepackage{makeidx}         
\usepackage{graphicx}        
\usepackage{multicol}        
\usepackage[bottom]{footmisc}
\usepackage{multirow}

\usepackage{amsmath}
\usepackage{amsfonts}
\usepackage{bm}
\usepackage{easyeqn}
\usepackage{xcolor}
\usepackage{algorithm}
\usepackage{algorithmicx}
\usepackage{algpseudocode}
\usepackage{hhline}

\makeindex             

\algrenewcommand\algorithmicrequire{\textbf{Input:}}
\algrenewcommand\algorithmicensure{\textbf{Output:}}

\def\T{\mathrm{T}}
\def\t{\times}

\def\nsize{n}
\def\RR{\mathbb{R}}
\def\EE{\mathbb{E}}
\def\PP{\mathbb{P}}
\def\eps{\varepsilon}
\def\cv{\mathbf{c}}
\def\ev{\mathbf{e}}
\def\fv{\mathbf{f}}
\def\gv{\mathbf{g}}
\def\nv{\mathbf{n}}
\def\uv{\mathbf{u}}

\def\alphav{\boldsymbol{\alpha}}
\def\ex{\mathrm{e}}
\def\identity{\mathbf{I}}

\def\nclusters{K}
\def\probMatrix{\mathbf{P}}
\def\adjacencyMatrix{\mathbf{A}}
\def\communityMatrix{\mathbf{B}}
\def\nodeCommunityMatrix{\pmb{\mathrm{\Theta}}}
\def\nodeWeights{\bm{\theta}}
\def\sparsityParam{\rho}
\def\probEigenvectors{\mathbf{U}}
\def\probEigenvalues{\mathbf{L}}
\def\adjacencyEigenvectors{\hat{\probEigenvectors}}
\def\adjacencyEigenvalues{\hat{\probEigenvalues}}
\def\noiseMatrix{\mathbf{N}}
\def\noiseVector{\nv}
\def\communityMatrixEstimate{\hat{\mathbf{B}}}
\def\nodeCommunityMatrixEstimate{\hat{\nodeCommunityMatrix}}

\def\nodeCommunitySet{\bar{\nodeCommunityMatrix}_{\nsize, \nclusters}}
\def\eigenMatrix{\mathbf{H}}
\def\orthMatrix{\mathbf{O}}
\def\errorAdjacency{\beta(\adjacencyMatrix, \probMatrix)}
\def\condNumber{\kappa}

\def\targetMatrix{\mathbf{G}}
\def\basisMatrix{\mathbf{F}}
\def\weightsMatrix{\mathbf{W}}
\def\noisyTargetMatrix{\tilde{\targetMatrix}}
\def\basisMatrixEstimate{\hat{\mathbf{F}}}

\newtheorem{condition}{Condition}

\begin{document}

\title*{Consistent Estimation of Mixed Memberships with Successive Projections}

\author{Maxim Panov, Konstantin Slavnov and Roman Ushakov}
\institute{Maxim Panov \at Skolkovo Institute of Science and Technology (Skoltech), Institute for Information Transmission Problems of RAS, Moscow, Russia, \email{m.panov@skoltech.ru}
\and
Konstantin Slavnov \at Skolkovo Institute of Science and Technology (Skoltech) \email{k.slavnov@skoltech.ru}
\and
Roman Ushakov \at Moscow Institute of Physics and Technology, Institute for Information Transmission Problems of RAS, Moscow, Russia, \email{ushakov.ra@phystech.edu}
}

\maketitle

\abstract{
  This paper considers the parameter estimation problem in Mixed Membership Stochastic Block Model (MMSB), which is a quite general instance of random graph model allowing for overlapping community structure. We present the new algorithm \textit{successive projection overlapping clustering} (SPOC) which combines the ideas of spectral clustering and geometric approach for separable non-negative matrix factorization. The proposed algorithm is provably consistent under MMSB with general conditions on the parameters of the model. SPOC is also shown to perform well experimentally in comparison to other algorithms.
}






\section{Introduction}
\label{sec: intro}
  Community detection is an important problem in modern network analysis. It has wide applications in analysis of social and biological networks~\cite{Girvan2002, Backstrom2006}, designing network protocols~\cite{Lu2015} and many other areas. Recently, much attention has been paid to detection of overlapping communities, where each node in a network may belong to multiple communities. Such a situation is quite common, and most prominent examples include overlapping communities in social networks~\cite{Leskovec2012}, where each user may belong to several social circles, and protein-protein interaction networks~\cite{Palla2005}, where a protein may belong to multiple protein complexes.
  
  One of the most widely used approaches for designing community detection algorithms (both for detection of overlapping and non-overlapping communities) can be summarized by the following general scheme:
  \begin{enumerate}
    \item Based on the adjacency matrix of a graph \(\adjacencyMatrix\), the embedding vectors \(\uv_{i}\) for the nodes are computed.

    \item The resulting embedding vectors \(\uv_{i}\) are clustered and the representative vectors \(\cv_{k}\) are found for each cluster.

    \item Certain post-processing is done, which determines for each node \(i\) to which communities it belongs based on the embedding vector \(\uv_{i}\) and community representatives \(\cv_{k}\).
  \end{enumerate}

  The step \((1)\) can be done in multiple ways, the most popular being spectral embeddings~\cite{Luxburg2007}, non-negative matrix factorization~\cite{Yang2013} and embeddings based on random walks~\cite{Perozzi2014}. The step \((2)\) is usually done via k-means or k-medians clustering with \(\cv_{k}\) being cluster centroids. For some methods the step \((2)\) is avoided and algorithm directly detects community affiliations from embedding vectors \(\uv_{i}\), see, for example,~\cite{Yang2013}. The step (3) is usually done by the decomposition of vector \(\uv_{i}\) in terms of basis vectors \(\cv_{k}\) and thresholding the coefficients of this decomposition.

  We note that the majority of overlapping community detection methods come with no guarantees on their performance. However, recently several approaches were proposed which consistently estimate parameters and detect overlapping communities in graphs under certain assumptions on the graph generation model, see, for example, OCCAM~\cite{Zhang2014}, SAAC~\cite{Kaufmann2016}, GeoNMF~\cite{Mao2017} and tensor-based approach~\cite{Anandkumar2013}. The models assumed in these works start from the classical Stochastic Block Model (SBM) and consider different generalizations which allow for overlapping community structure. All these methods follow the aforementioned general scheme, however have their own peculiarities and limitations. For example, SAAC assumes that specific node either belongs or not to the particular community while other methods assume more general community membership weights which are supposed to be real numbers from \([0, 1]\). OCCAM method has a very general model, however, it comes with certain conditions for consistent parameter recovery which seem to be rarely satisfied. GeoNMF algorithm considers MMSB model, but concentrates on the limited situation, where the nodes can have an edge between them only if they belong to the same community, while inter-community edges are prohibited. Finally, tensor-based method of~\cite{Anandkumar2013} is built for general MMSB model, but its high computational complexity limits applications to large graphs. Also all these algorithms come with certain parameters which do not have a clear guidelines for selection, except some suggestions on the asymptotic order of the parameter.

  In this work, we propose a new algorithm for parameter estimation in Mixed Membership Stochastic Block Model (MMSB)~\cite{Airoldi2008}, called \textit{successive projection overlapping clustering} (SPOC). The algorithm starts from the spectral embedding based on the adjacency matrix of the graph, then finds nearly pure nodes via successive projection algorithm (SPA)~\cite{Araujo2001} and finally reconstructs community memberships via least-squares fit. SPOC has following important features:
  \begin{enumerate}
    \item The algorithm consistently estimates parameters for general variant of MMSB, where nodes can belong to communities with continuous weights (from \([0, 1]\)) and all the communities can generate edges between each other.
    \item The algorithm has no input parameters except number of clusters.
    \item The algorithm is computationally efficient with complexity dominated by the SVD of the adjacency matrix.
    \item Empirically SPOC shows the better performance in wide range of problems in comparison with other algorithms for parameter estimation in MMSB and related models.
  \end{enumerate}

  The paper is structured as follows. In Section~\ref{sec: model} we introduce the MMSB model, compare it with other related models from the literature and discuss the identifiability of the parameters in this model. In Section~\ref{sec: algo} we introduce the new SPOC algorithm and discuss the intuition behind it. In Section~\ref{sec: consistency} we prove that SPOC consistently estimates the parameters of the MMSB. Section~\ref{sec: experiments} describes the experimental comparison of SPOC and other algorithms on simulated and real data. Finally, some conclusive remarks are made in Section~\ref{sec: conclusions}.

\section{Mixed membership stochastic block model (MMSB)}
\label{sec: model}

\subsection{The model}
  Let us introduce the basic model we are going to work with. We assume that we observe symmetric binary matrix \(\adjacencyMatrix\) of size \(\nsize\). Each \(A_{ij}\) for \(1 \le i < j \le \nsize\) is an independent Bernoulli random variable with respective parameter \(P_{ij} \in [0, 1]\), which form symmetric matrix \(\probMatrix \in [0, 1]^{\nsize \t \nsize}\). In the matrix form we can write it as:
  \begin{EQA}[c]
    \adjacencyMatrix \sim \mathrm{Bernoulli}(\probMatrix).
  \end{EQA}
  We note that \(\adjacencyMatrix\) can be considered as the adjacency matrix of the random graph and further assume that there are \(\nclusters\) communities in the graph. The mixed membership stochastic block model (MMSB) assumes that \(P_{ij} = \nodeWeights_{i} \communityMatrix \nodeWeights_{j}^{\T}\) for \(1 \le i < j \le \nsize\). Here \(\communityMatrix \in [0, 1]^{\nclusters \t \nclusters}\) is a symmetric matrix of community-community edge probabilities, which element \(B_{kl}\) is a probability of an edge between nodes from communities \(k\) and \(l\). The row vector \(\nodeWeights_{i} \in [0, 1]^{\nclusters}\) is a community membership vector for node \(i\). We introduce community membership matrix \(\nodeCommunityMatrix \in [0, 1]^{\nsize \t \nclusters}\) and further assume that each row \(\nodeWeights_{i}\) of \(\nodeCommunityMatrix\) is normalized \(\sum_{k = 1}^{\nclusters} \theta_{ik} = 1\). So, we can interpret \(\nodeWeights_{i}\) as a vector of probabilities for the node \(i\) to belong to one of the communities.
  Finally, in the matrix form we can write
  \begin{EQA}[c]
  \label{eq: mmsbDef}
    \probMatrix = \nodeCommunityMatrix \communityMatrix \nodeCommunityMatrix^{\T}.
  \end{EQA}
  Let us further denote
  \begin{EQA}[c]
    \nodeCommunitySet = \bigl\{\nodeCommunityMatrix \in [0, 1]^{\nsize \t \nclusters}\colon \sum_{k = 1}^{\nclusters} \theta_{ik} = 1, ~ i = 1, \dots, \nsize\bigr\}.
  \end{EQA}

  The considered model is directly related to several models in the literature.
  We note that compared to original definition of MMSB~\cite{Airoldi2008} we do not assume Dirichlet distribution of community memberships \(\nodeWeights_{i}\). The other related models are OCCAM~\cite{Zhang2014}, where different normalization of community membership vectors is considered, and SBMO~\cite{Kaufmann2016}, where only binary community memberships are allowed. Compared to the variant of MMSB considered in~\cite{Mao2017} we consider more general situation, where matrix \(\communityMatrix\) is allowed to be any full rank symmetric matrix. Finally, the ordinary stochastic block model is particular instance of our model, where each vector of community memberships \(\nodeWeights_{i}\) has exactly one non-zero entry (equal to one).

\vspace{-20pt}

\subsection{Identifiability}
  In general, the models of type~\eqref{eq: mmsbDef} are not identifiable and certain conditions are needed to ensure identifiability. The identifiability issue is due to the fact that there might be different pairs of matrices \(\communityMatrix\) and \(\nodeCommunityMatrix\) which generate the same matrix \(\probMatrix\), see related discussion and examples of non-identifiability in~\cite{Kaufmann2016}. We note that in the considered setting the indices of communities are not identifiable and thus can be recovered only up to permutation.

  We impose the following conditions which make the MMSB identifiable.
  \begin{condition}[Identifiability]
  \label{cond: indentifiability}
    \text{}
    \begin{enumerate}
      \item There is at least one ``pure'' node at each community, i.e. for each \(k = 1, \dots, \nclusters\) there exists \(i\) such that \(\theta_{ik} = \sum_{l = 1}^{\nclusters} \theta_{il} = 1\).
      \item Matrix \(\communityMatrix \in [0, 1]^{\nclusters \t \nclusters}\) is full rank.
      \item \(\nodeCommunityMatrix \in \nodeCommunitySet\), i.e. every row of matrix \(\nodeCommunityMatrix\) sums to 1: \(\sum_{k = 1}^{\nclusters} \theta_{ik} = 1, ~ i = 1, \dots, \nsize\).
    \end{enumerate}
  \end{condition}

  It appears that these conditions are sufficient for the identifiability of MMSB, see the following theorem.

  \begin{theorem}
  \label{theorem: identifiability}
    If the Condition~\ref{cond: indentifiability} is satisfied then MMSB model~\eqref{eq: mmsbDef} is identifiable up to simultaneous permutation of rows and columns in matrix \(\communityMatrix\) and columns in matrix \(\nodeCommunityMatrix\).
  \end{theorem}
  The Condition~\ref{cond: indentifiability} may seem quite strict.
  However, we note that if the matrix \(\communityMatrix\) is not full rank then there might be multiple matrices \(\nodeCommunityMatrix\), which give the same matrix \(\probMatrix\). Some normalization condition on matrices \(\communityMatrix\) and  \(\nodeCommunityMatrix\) is needed to set the scale of one matrix and make the scale of the other matrix identifiable. The particular condition \(\nodeCommunityMatrix \in \nodeCommunitySet\) is chosen for the ease of probabilistic interpretation, while other conditions can be considered (leading to models formally different from MMSB). Finally, the condition on existence of ``pure'' nodes is the most tricky one and is not necessarily satisfied in the real life applications. However, while it is also not necessary for identifiability, the possible alternative conditions for the identifiability are still quite strict, see discussion in~\cite{Huang2016}.

\vspace{-20pt}

\section{Algorithm}
\label{sec: algo}
  The SPOC algorithm general scheme can be summarized as follows.

  \begin{algorithm}
  \label{algo: SPOC}
    \caption{SPOC}
    \begin{algorithmic}[1]
      \Require{Adjacency matrix \(\adjacencyMatrix\) and number of communities \(\nclusters\).}
      \Ensure{Estimated community-community edge probability \(\communityMatrixEstimate\) and community membership \(\nodeCommunityMatrixEstimate\) matrices.}
      \State Get the rank-\(\nclusters\) eigenvalue decomposition \(\adjacencyMatrix \simeq \adjacencyEigenvectors \adjacencyEigenvalues \adjacencyEigenvectors^{\T}\).
      \State Run SPA algorithm with input \((\adjacencyEigenvectors^{\T}, \nclusters)\), which outputs set of indices \(J\) of cardinality \(\nclusters\).
      \State \(\basisMatrixEstimate = \adjacencyEigenvectors[J, :]\).
      \State \(\communityMatrixEstimate = \basisMatrixEstimate \adjacencyEigenvalues \basisMatrixEstimate^{\T}\).
      \State \(\nodeCommunityMatrixEstimate = \adjacencyEigenvectors \basisMatrixEstimate^{\T} (\basisMatrixEstimate \basisMatrixEstimate^{\T})^{-1}\).
    \end{algorithmic}
  \end{algorithm}
  The only unspecified part of the algorithm is the application of successive projection algorithm (SPA) to the matrix \(\adjacencyEigenvectors^{\T}\). We will briefly describe this algorithm in Section~\ref{sec: SPA} below, see also the detailed discussions in~\cite{Gillis2014,Gillis2015,Mizutani2016}.

\subsection{Adjacency matrix decomposition}
  An important first step of the algorithm is decomposition of the adjacency matrix in a form \(\adjacencyMatrix \simeq \adjacencyEigenvectors \adjacencyEigenvalues \adjacencyEigenvectors^{\T}\), where \(\adjacencyEigenvalues\) is the \(\nclusters \t \nclusters\) diagonal matrix containing \(\nclusters\) leading eigenvalues of \(\adjacencyMatrix\) and \(\adjacencyEigenvectors\) is the \(\nsize \t \nclusters\) orthogonal matrix of corresponding eigenvectors. We note that we can in parallel consider the eigen decomposition of matrix \(\probMatrix = \probEigenvectors \probEigenvalues \probEigenvectors^{\T}\), where diagonal matrix \(\probEigenvalues \in \RR^{\nclusters \t \nclusters}\) and orthogonal matrix \(\probEigenvectors \in \RR^{\nsize \t \nclusters}\) are population counterparts of matrices \(\adjacencyEigenvalues\) and \(\adjacencyEigenvectors\).

\subsection{Separable noisy matrix factorization}
\label{sec: SPA}
  Now our goal is to compute estimates for matrices \(\communityMatrix\) and \(\nodeCommunityMatrix\) based on the matrix \(\adjacencyEigenvectors\). We can represent matrix \(\adjacencyEigenvectors\) in the following way:
  \begin{EQA}[c]
    \adjacencyEigenvectors = \nodeCommunityMatrix \basisMatrix + \noiseMatrix,
  \label{eq: nmf}
  \end{EQA}
  where \(\basisMatrix\) is a matrix such that \(\probEigenvectors = \nodeCommunityMatrix \basisMatrix\) and \(\noiseMatrix \in \RR^{\nsize \t \nclusters}\) is a matrix of noise due to the approximation of the matrix \(\probEigenvectors = \nodeCommunityMatrix \basisMatrix\) by an empirical counterpart \(\adjacencyEigenvectors\).

  Due to normalization assumption on matrix \(\nodeCommunityMatrix \in \nodeCommunitySet\) linear combinations \(\nodeWeights_{i} \basisMatrix\) lie in the simplex with vertices corresponding to rows of matrix \(\basisMatrix\). Thus, the matrix factorization problem~\eqref{eq: nmf} is a particular instance of so-called noisy separable non-negative matrix factorization which was extensively studied in the literature, see \cite{Arora2012,Gillis2014,Mizutani14,Li2016} for the examples of provably efficient algorithms for this problem. \textit{Separability} here means the existence of ``pure'' nodes in terms of MMSB model.

  The main idea for the whole family of algorithms called \textit{successive projective estimation} (SPA)~\cite{Araujo2001} is to iteratively find the rows of matrix \(\adjacencyEigenvectors\) with maximum norm and project on the subspace orthogonal to these rows. The correctness of the algorithm in noiseless case bases on the fact that any strongly convex function attains its maximum in one of the basis vertices of simplex, which means that we iteratively detect the set of ``pure'' nodes for all the communities. In the noisy case, certain conditions are needed for the noise level to ensure that nearly ``pure'' nodes will be extracted, see the precise statement in Section~\ref{sec: consistencySPA}.
  In SPOC algorithm, we use the variant of SPA algorithm from~\cite{Mizutani14}, which performs additional preconditioning before running actual SPA procedure.  

  We finally note that if the rows of matrix \(\nodeCommunityMatrix\) have some general distribution then these rows won't concentrate around ``pure'' nodes, which prohibits the direct application of clustering algorithms like k-means to this problem as it was used in OCCAM algorithm~\cite{Zhang2014}. The approach to overcome this difficulty was proposed in GeoNMF algorithm~\cite{Mao2017} which filters the ``intermediate'' nodes and leaves the nodes sufficiently close to pure nodes. However, it is unclear how to choose the parameter of GeoNMF which governs the filtering threshold in practice.

\subsection{Post-processing}
  We note that some elements in matrices \(\communityMatrixEstimate\) and \(\nodeCommunityMatrixEstimate\) may be negative or greater than \(1\). While these estimates are still consistent as we will see in Section~\ref{sec: consistency}, for the practical usage we threshold elements of matrices \(\communityMatrixEstimate\) and \(\nodeCommunityMatrixEstimate\) to be between \(0\) and \(1\). Obviously, this can only improve the consistency properties of the estimates. Importantly, we do not perform any normalization for matrix \(\nodeCommunityMatrixEstimate\), so that \(\nodeCommunityMatrixEstimate\) is only asymptotically close to \(\nodeCommunitySet\), but doesn't belong to it exactly. The finite sample performance of the algorithm might be improved if some variant of normalized estimate for \(\nodeCommunityMatrix\) is considered. Finally, we note that one can conduct the community detection based on the estimated parameters. The simplest possible way is to report that the node belongs to the community if corresponding community membership exceeds some prespecified threshold.

\section{Provable guarantees for SPOC}
\label{sec: consistency}
  In this section, we are going to provide theoretical guarantees assuring that estimates \(\communityMatrixEstimate\) and \(\nodeCommunityMatrixEstimate\) concentrate around corresponding population parameters \(\communityMatrix\) and \(\nodeCommunityMatrix\). Certain assumptions are needed for our analysis. We will assume, that the condition number \(\condNumber(\communityMatrix)\) of matrix \(\communityMatrix\) is fixed while the value \(\sparsityParam = \max_{i,j} \communityMatrix_{i,j}\) is allowed to change with the sample size. For the matrix \(\nodeCommunityMatrix\) the most natural way is to assume that its rows are random vectors from some distribution on~\(\nodeCommunitySet\). The well known example of such approach is the original variant of MMSB model introduced in~\cite{Airoldi2008}, where community memberships follow the Dirichlet distribution: \(\nodeWeights_{i} \sim \mathrm{Dirichlet}(\alphav)\) for some \(\alphav \in \RR_{+}^{\nclusters}\). In our analysis, we will consider more general situation and assume that \(\nodeWeights_{i}\) are i.i.d. samples from some general distribution \(\PP_{\nodeWeights}\) on~\(\nodeCommunitySet\). More specifically, we will require the following condition.

  \begin{condition}[Community memberships distribution]
  \label{cond: community memberships}
    \text{}
    
    Community membership vectors \(\nodeWeights_{i}\) are i.i.d. samples from the distribution \(\PP_{\nodeWeights}\) on \(\nodeCommunitySet\), which has non-zero mass in all ``pure'' nodes.
  \end{condition}

  Our goal is to study the properties of the estimates in case when the community memberships follow the model above. The main result is summarized in the next theorem. 
  \begin{theorem}
  \label{theorem: mainBound}
    Let's consider the model~\eqref{eq: mmsbDef} with matrix \(\communityMatrix\) being full rank. Let the Condition~\ref{cond: community memberships} is satisfied. Let SPOC algorithm outputs matrices \(\nodeCommunityMatrixEstimate\) and \(\communityMatrixEstimate\). Then there exist constants \(c\) and \(C\) depending only on the condition numbers of the matrices \(\communityMatrix\) and \(\nodeCommunityMatrix\) and parameter \(r > 0\) such that for \(\sparsityParam \ge c \frac{\log \nsize}{\nsize}\) it holds with probability at least \(1 - \nsize^{-r}\) that
    \begin{EQA}[c]
    \label{eq: communityFinal}
      \frac{\bigl\|\communityMatrixEstimate - \mathrm{\Pi}_{\basisMatrix} \communityMatrix \mathrm{\Pi}_{\basisMatrix}^{\T}\bigr\|_{F}}{\|\communityMatrix\|_{F}}
      \le
      C \nclusters \sqrt{\frac{\log \nsize}{\sparsityParam^2 \nsize}}
    \end{EQA}
    and
    \begin{EQA}[c]
    \label{eq: nodeFinal}
      \frac{\bigl\|\nodeCommunityMatrixEstimate - \nodeCommunityMatrix \mathrm{\Pi}_{\basisMatrix}^{\T}\bigr\|_{F}}{\|\nodeCommunityMatrix\|_{F}}
      \le
      C \nclusters \sqrt{\frac{\log \nsize}{\sparsityParam^2 \nsize}},
    \end{EQA}
    where \(\mathrm{\Pi}_{\basisMatrix}\) is some permutation matrix.
  \end{theorem}
  The bounds~\eqref{eq: communityFinal} and~\eqref{eq: nodeFinal} show, that SPOC algorithm provides accurate estimates of MMSB model parameters with high probability.

  \begin{remark}
    We expect, that there should exist different algorithm which can improve the rate in bound~\eqref{eq: communityFinal} by \(\sqrt{\nsize}\). Also, it is likely to be possible to improve both 
    bounds by \(\sqrt{\sparsityParam}\) using more elaborate analysis of spectral properties.
  \end{remark}

\section{Experiments}
\label{sec: experiments}
  We conducted the series of experiments on both simulated and real data to assess the quality of results obtained by SPOC algorithm and compare it to other algorithms for detection of parameters in MMSB and related models.

\subsection{Simulated data}
  We generate the rows of matrix \(\nodeCommunityMatrix\) from Dirichlet distribution and add some number of pure nodes to ensure the identifiability condition. Default parameter settings were: number of nodes \(\nsize = 1000\), number of communities \(\nclusters = 3\), pure nodes number \(3\), Dirichlet parameter \(\alpha = 0.5\) and \(\communityMatrix = \mathrm{diag}(0.3,\, 0.5,\, 0.7)\). Each experiment was repeated \(10\) times and results were averaged over runs.

  We considered the series of experiments varying one of the parameters in each of them.

  \begin{enumerate}
    \item We varied the number of nodes in the graph \(\nsize \in [1000, 5000]\).
    \item We varied the diagonal elements of matrix \(\communityMatrix\) making it skewed: \(\communityMatrix = \mathrm{diag}(0.5 - \eps,\, 0.5,\, 0.5 + \eps)\) for \(\eps \in [0.05, 0.45]\).
    \item We varied the Dirichlet distribution parameter \(\alpha \in [0.5, 4]\).
    \item We started from default diagonal matrix \(\communityMatrix\) and varied off-diagonal elements in range \([0, 0.4]\).
  \end{enumerate}
  We experimentally compared our proposed algorithm SPOC with GeoNMF algorithm~\cite{Mao2017}. We did not compare SPOC with other algorithms for parameter estimation in MMSB model~\cite{Airoldi2008,Anandkumar2013,Mao2017} as these algorithms were significantly outperformed by GeoNMF according to~\cite{Mao2017}.  We report the relative error of estimation for both parameters \(\communityMatrix\) and \(\nodeCommunityMatrix\).

  \begin{figure}[!ht]
    \centering
    \includegraphics[width=\textwidth]{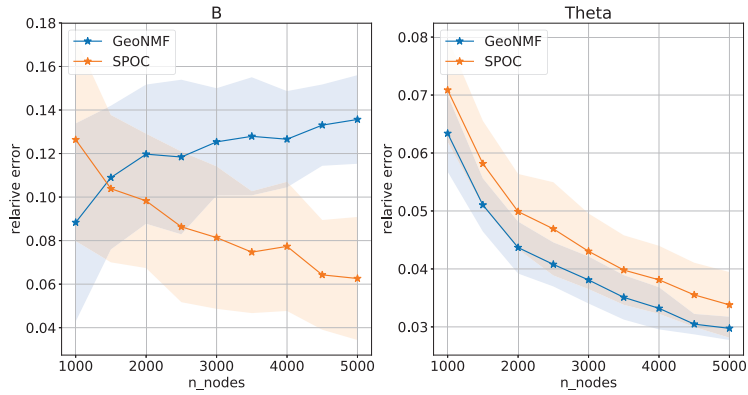}
    \caption{Experiment with varying number of nodes~\(\nsize\).}
    \label{fig:Varying_nodes_number}
  \end{figure}

	\begin{figure}[!ht]
    \centering
	  \includegraphics[width=\textwidth]{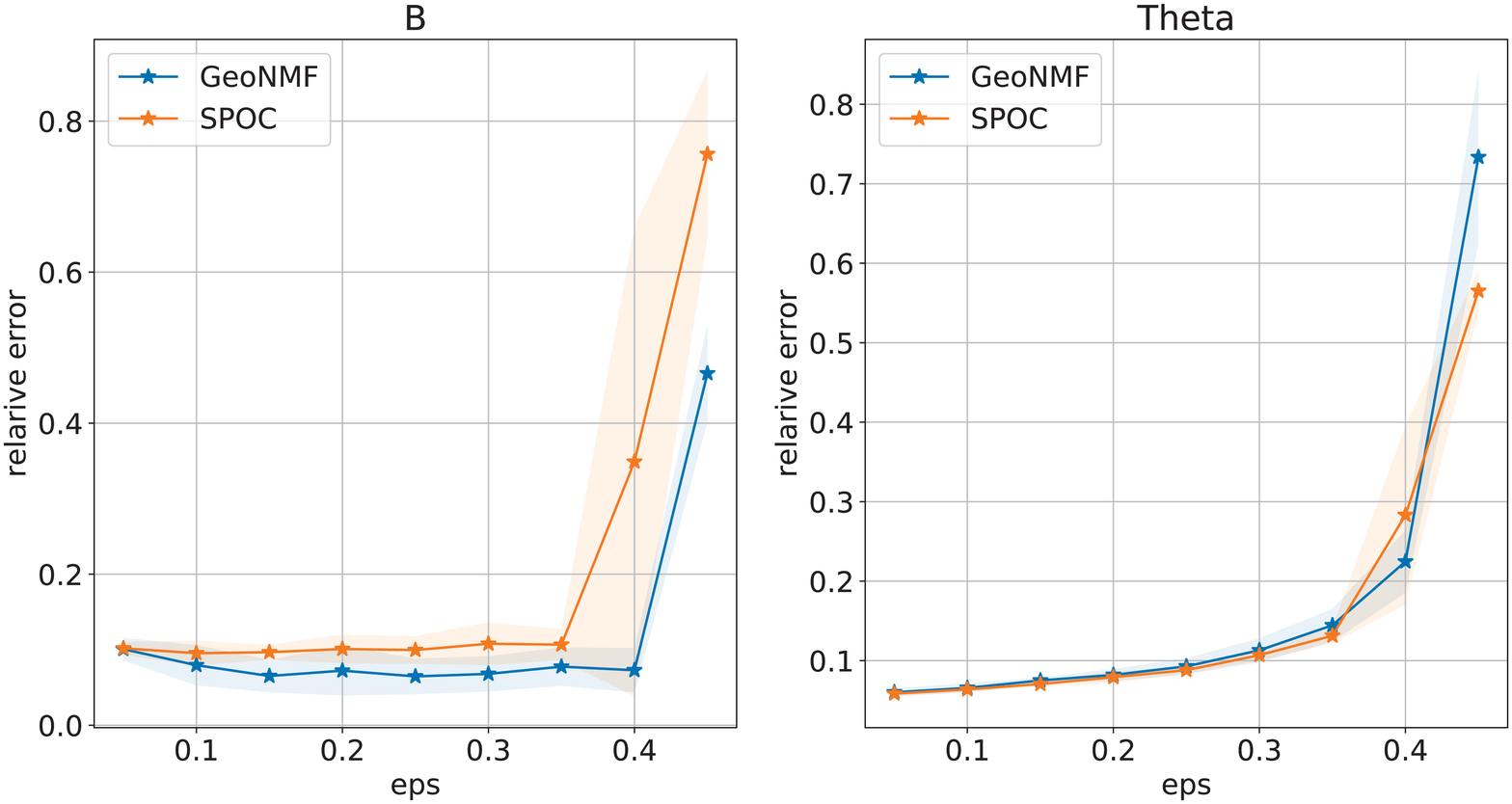}
	  \caption{Experiment with skewed $\communityMatrix$ matrix.}
	  \label{fig:Skewed_B}
  \end{figure}

  \begin{figure}[!ht]
    \centering
    \includegraphics[width=\textwidth]{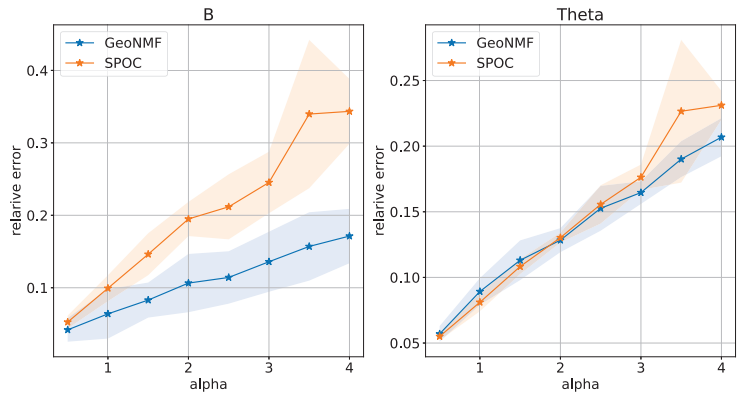}
    \caption{Experiment with varying parameter \(\alphav\) of Dirichlet distribution.}
    \label{fig:Varying_Dirichlet_alpha}
  \end{figure}

  \begin{figure}[!ht]
    \centering
  	\includegraphics[width=\textwidth]{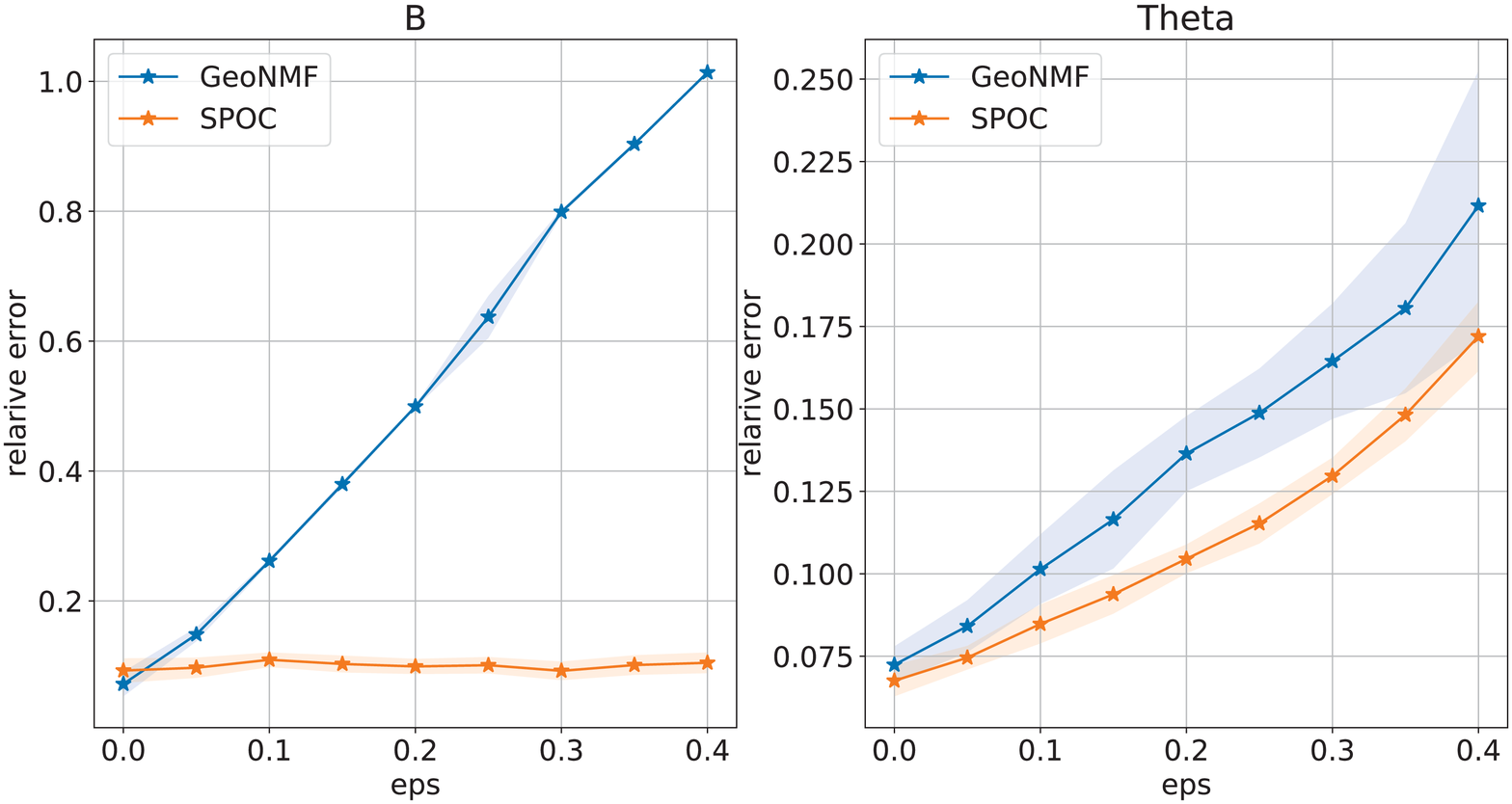}
  	\caption{Experiment with noisy off-diagonal elements of $\communityMatrix$.}
  	\label{fig:Noisy_off-diagagonal_B}
	\end{figure}
  The results of experiments are presented on Figures~\ref{fig:Varying_nodes_number}--\ref{fig:Varying_Dirichlet_alpha}. We note that in the first 3 experiments GeoNMF algorithm is expected to have advantage over SPOC as it assumes the diagonal structure of matrix \(\communityMatrix\). However, we see that this advantage is not significant in most cases for estimation of \(\communityMatrix\), while for \(\nodeCommunityMatrix\) the considered methods show very similar performance. Surprisingly, GeoNMF performance is not improving with the growth of the graph, which is not the case for SPOC. In the last experiment, SPOC outperforms GeoNMF as it is based on more general (non-diagonal) structure of matrix \(\communityMatrix\).

\subsection{Real data}
  Finally, we tested the considered methods on the co-authorship networks created from DBLP and from the Microsoft Academic Graph by~\cite{Mao2017}. In these data, nodes correspond to authors and ground truth community memberships \(\nodeWeights_i\) are determined normalizing the number of papers published by the author in a subfield. We refer to~\cite{Mao2017} for the detailed description of data preprocessing.
  The considered networks have the following subfields:
  \begin{itemize}
    \item DBLP1: Machine Learning, Theoretical Computer Science, Data Mining, Computer Vision, Artificial Intelligence, Natural Language Processing;
    \item DBLP2: Networking and Communications, Systems, Information Theory;
    \item DBLP3: Databases, Data Mining, World Wide Web;
    \item  DBLP4: Programming Languages, Software Engineering, Formal Methods;
    \item DBLP5: Computer Architecture, Computer Hardware, Real-time and Embedded Systems, Computer-aided Design;
    \item MAG1: Computational Biology and Bioinformatics, Organic Chemistry, Genetics;
    \item MAG2: Machine Learning, Artificial Intelligence, Mathematical Optimization.
  \end{itemize}
  We use average Spearman correlation coefficient between actual and predicted community memberships as a quality measure. The results are summarized in Figure~\ref{fig:real_data}. We note that either GeoNMF or SPOC show best results for all datasets. However, all the algorithms show very limited performance on all the considered problems.

  \begin{figure}[!ht]
    \centering
    \includegraphics[width=\textwidth]{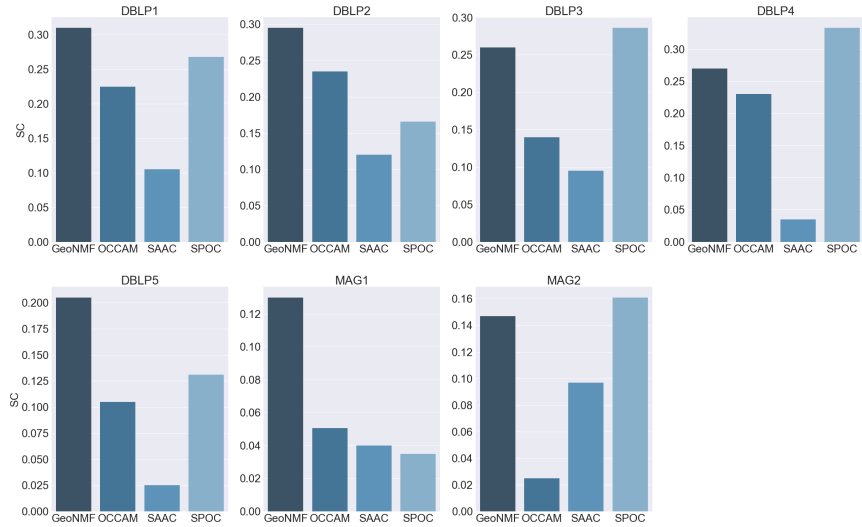}
    \caption{Experiments on DBLP and MAG co-authorship networks.}
    \label{fig:real_data}
  \end{figure}

\vspace{-10pt}

\section{Conclusions}
\label{sec: conclusions}
  In this work, we consider the problem of parameter estimation in Mixed Membership Stochastic Block Model (MMSB), which is directly related to the problem of overlapping community detection. We present the new algorithm \textit{successive projection overlapping clustering} (SPOC) which combines the ideas of spectral clustering and geometric approach to parameter estimation in separable non-negative matrix factorization. The proposed algorithm is provably consistent under MMSB with general conditions on the parameters of the model. SPOC is also shown to perform well experimentally in comparison to other algorithms.

  The work leaves several important open questions including the lower bounds for the considered problem over the certain subclass of identifiable MMSB's and the possibility to propose the algorithm with improved upper bound for the estimate of matrix \(\communityMatrix\). Also the more detailed experimental comparison is needed on real world networks which allow for good quality of community detection.


\begin{acknowledgement}
  The research was supported by the Russian Science Foundation grant (project 14-50-00150). The authors would like to thank Nikita Zhivotovskiy and Alexey Naumov for very insightful discussions on matrix concentration. The help of Emilie Kaufmann, who provided the code of SAAC algorithm, is especially acknowledged.
\end{acknowledgement}

\vspace{-30pt}

\bibliography{overlapping}
\bibliographystyle{plain}

\newpage

\section{Consistency analysis}
\label{sec: consistency_appendix}
\subsection{Concentration of spectral embedding}
  The main step in analyzing consistency of our algorithm is consistency of estimation of matrix \(\probEigenvectors = \nodeCommunityMatrix \basisMatrix\) by \(\adjacencyEigenvectors\). We note that the eigenvectors can be identified up to some rotation defined by orthogonal matrix \(\orthMatrix_{\probMatrix}\). The difference \(\adjacencyEigenvectors - \probEigenvectors \orthMatrix_{\probMatrix}\) can be bounded in Frobenious norm, however the error bound of SPA algorithm (see Section~\ref{sec: consistencySPA} for details) depends on maximum of norms for rows of matrix \(\adjacencyEigenvectors - \probEigenvectors \orthMatrix_{\probMatrix}\), which is of smaller order. The following lemma gives a bound on the distance between rows of \(\adjacencyEigenvectors\) and \(\probEigenvectors \orthMatrix_{\probMatrix}\).

  \begin{lemma}
    \label{lemma: rowFactorBound}
    Assume that \(\probMatrix \in \RR^{\nsize \t \nsize}\) is a rank \(\nclusters\) symmetric matrix with smallest non-zero singular value \(\lambda_{\nclusters}(\probMatrix)\). Let \(\adjacencyMatrix\) be any symmetric matrix such that \(\|\adjacencyMatrix - \probMatrix\| \le \frac{1}{2} \lambda_{\nclusters}(\probMatrix)\) and \(\adjacencyEigenvectors, \probEigenvectors\) are the \(\nsize \t \nclusters\) matrices of eigenvectors for matrices \(\adjacencyMatrix\) and \(\probMatrix\) corresponding to top-\(\nclusters\) eigenvalues. Then
    \begin{EQA}
    \label{eq: rowConcentr}
      \|\ev_{i}^{\T}(\adjacencyEigenvectors - \probEigenvectors \orthMatrix_{\probMatrix})\|_{F}
      &\le&
      23 \nclusters^{1/2} \condNumber(\probMatrix) ~ \frac{\|\ev_{i}^{\T} \adjacencyMatrix\|_{F}\cdot \|\adjacencyMatrix - \probMatrix\|}{\lambda_{\nclusters}^{2}(\probMatrix)}
      +
      \frac{\|\ev_{i}^{\T}(\adjacencyMatrix - \probMatrix) \probEigenvectors\|_{F}}{\lambda_{\nclusters}(\probMatrix)},
    \end{EQA}
    where \(\ev_{i}\) is a vector of length \(\nsize\) with \(1\) in the \(i\)-th position and \(\orthMatrix_{\probMatrix}\) is some orthogonal matrix.
  \end{lemma}
  This lemma may seem rather technical, however it shows that the right hand side has the terms, which are projections of vector of bounded random variables onto span of \(\nclusters\) orthogonal vectors, which can be bounded better then just by multiple of matrix norms, see Theorem~\ref{theorem: mainBound} below.

  Let us denote the right hand side of~\eqref{eq: rowConcentr} by
  \begin{EQA}[c]
    \beta_{i}(\adjacencyMatrix, \probMatrix)
    =
    23 \nclusters^{1/2} \condNumber(\probMatrix) ~ \frac{\|\ev_{i}^{\T} \adjacencyMatrix\|_{F}\cdot \|\adjacencyMatrix - \probMatrix\|}{\lambda_{\nclusters}^{2}(\probMatrix)}
    +
    \frac{\|\ev_{i}^{\T}(\adjacencyMatrix - \probMatrix) \probEigenvectors\|_{F}}{\lambda_{\nclusters}(\probMatrix)}
  \end{EQA}
  and also let's define
  \begin{EQA}[c]
  \label{eq: errorDef}
    \errorAdjacency = \max_{i \in \{1, \dots, \nsize\}} \beta_{i}(\adjacencyMatrix, \probMatrix).
  \end{EQA}

\subsection{Noisy separable matrix factorization}
\label{sec: consistencySPA}
  Now we are going to give the bound on the error of the matrix factorization in separable case, i.e. the solution of the following problem:
  \begin{EQA}[c]
    \targetMatrix = \basisMatrix \weightsMatrix \text{~for~} \basisMatrix \in \RR^{r \t \nclusters}, \weightsMatrix = (I, M) \mathrm{\Pi} \in \RR_{+}^{\nclusters \t \nsize},
  \end{EQA}
  where \(I \in \RR^{\nclusters \t \nclusters}\) is an identity matrix, \(M \in \RR_{+}^{\nclusters \t (\nsize - \nclusters)}\) and \(\mathrm{\Pi} \in \RR^{\nsize \t \nsize}\) is a permutation matrix. We expect that we observe 
  \begin{EQA}[c]
    \noisyTargetMatrix = \targetMatrix + \noiseMatrix = \basisMatrix \weightsMatrix + \noiseMatrix,
  \end{EQA}
  where \(\noiseMatrix \in \RR^{\nclusters \t \nsize}\) is a perturbation (noise) matrix.

  The following theorem can be proved for the preconditioned SPA algorithm, see~\cite{Gillis2015,Mizutani2016}.
  \begin{theorem}
  \label{theorem: spaBasic}
    Let \(\targetMatrix = \basisMatrix \weightsMatrix\) and \(\noisyTargetMatrix = \targetMatrix + \noiseMatrix\). Suppose that \(\nclusters \geq 2\) and the Condition~\ref{cond: indentifiability} is satisfied. If in matrix \(\noiseMatrix\) each column \(\noiseVector_{i}\) satisfies \(\|\noiseVector_{i}\|_{F} \le \eps\) with
    \begin{EQA}[c]
      \eps \leq \frac{\lambda_{min}(\basisMatrix)}{1225 \sqrt{r}},
    \end{EQA}
    then SPA algorithm with the input \((\noisyTargetMatrix, r)\) returns the set of indices \(J\) such that there exists a permutation \(\pi\) which gives
    \begin{EQA}[c]
      \|\tilde{\gv}_{J(j)} - \fv_{\pi(j)}\|_2 \le (432 \condNumber(\basisMatrix) + 4) \eps
    \end{EQA}
    for all \(j = 1, \dots, r\), where \(\tilde{\gv}_{k}\) and \(\fv_{k}\) are the columns of matrices \(\noisyTargetMatrix\) and \(\basisMatrix\) correspondingly. Here we denote by \(\condNumber(\basisMatrix) = \frac{\lambda_{max}(\basisMatrix)}{\lambda_{min}(\basisMatrix)}\) is the condition number of the matrix \(\basisMatrix\).
  \end{theorem}
  We note that this error bound depends on the upper bound on individual errors \(\|\noiseVector_{i}\|\). From statistical point of view one might expect, that there should be an algorithm, which improves over this error bound if there are many ``pure'' columns in the matrix \(\targetMatrix\) so that the value of the error is diminished by averaging. However, to the best of our knowledge, no such algorithm complemented with the performance analysis can be found in the literature.

  Now we can reformulate the result of Theorem~\ref{theorem: spaBasic} for our particular situation with \(\targetMatrix = \adjacencyEigenvectors^{\T}\) and \(r = \nclusters\).
  \begin{corollary}
  \label{corollary: consistencyBasis}
    Let us consider the model~\eqref{eq: nmf} and let the Condition~\ref{cond: indentifiability} holds. Let also \(\|\adjacencyMatrix - \probMatrix\| \le \frac{1}{2} \lambda_{\nclusters}(\probMatrix)\). Then, the SPA algorithm with input \((\adjacencyEigenvectors^{\T}, \nclusters)\) returns the output set of indices \(J\) and corresponding matrix \(\basisMatrixEstimate = \adjacencyEigenvectors[J, :]\) such that there exist constants \(c_1\) and \(C_0\), and permutation \(\pi\), which ensure
    \begin{EQA}[c]
    \label{eq: consistencyBasis}
      \bigl\|\hat{\fv}_{j} - \fv_{\pi(j)}\bigr\|_2 \le C_0 \condNumber(\basisMatrix) \errorAdjacency
    \end{EQA}
    for all \(j = 1, \dots, \nclusters\) if the condition \(\errorAdjacency \leq c_1 \frac{\lambda_{min}(\basisMatrix)}{\nclusters^{1/2}}\) is satisfied.

    Consequently,
    \begin{EQA}[c]
    \label{eq: factorBound}
      \bigl\|\basisMatrixEstimate - \mathrm{\Pi}_{\basisMatrix} \basisMatrix \orthMatrix_{\probMatrix}\bigr\|_{F} \le C_0 \nclusters^{1/2} \condNumber(\basisMatrix) \errorAdjacency,
    \end{EQA}
    where \(\mathrm{\Pi}_{\basisMatrix}\) is a permutation matrix corresponding to the permutation \(\pi\).
  \end{corollary}

\subsection{Consistency of parameter estimates}
  Now we are ready to state the results on consistency of parameter estimates by SPOC algorithm.
  Based on inequality~\eqref{eq: factorBound} it is straightforward to get the error bound for an estimate \(\communityMatrixEstimate = \basisMatrixEstimate \adjacencyEigenvalues \basisMatrixEstimate^{\T}\) of matrix \(\communityMatrix\).

  \begin{theorem}
  \label{theorem: communityMatrixBound}
    Let us consider the model~\eqref{eq: nmf} and let the Condition~\ref{cond: indentifiability} holds. Let also \(\|\adjacencyMatrix - \probMatrix\| \le \frac{1}{2} \lambda_{\nclusters}(\probMatrix)\). Then SPOC algorithm outputs matrix \(\communityMatrixEstimate\) such that it holds
    \begin{EQA}[c]
      \bigl\|\communityMatrixEstimate - \mathrm{\Pi}_{\basisMatrix} \communityMatrix \mathrm{\Pi}_{\basisMatrix}^{\T}\bigr\|_{F}
      \le
      C \nclusters^{1/2} \condNumber(\nodeCommunityMatrix) \frac{\|\probMatrix\|}{\lambda_{\nclusters}(\nodeCommunityMatrix)} \errorAdjacency + C \nclusters^{1/2} \condNumber(\probMatrix) \frac{\|\adjacencyMatrix - \probMatrix\|}{\lambda_{\nclusters}^2(\nodeCommunityMatrix)},
    \end{EQA}
    where \(\condNumber(\nodeCommunityMatrix)\) and \(\condNumber(\probMatrix)\) are the condition numbers of matrices \(\nodeCommunityMatrix\) and \(\probMatrix\) respectively, \(\mathrm{\Pi}_{\basisMatrix}\) is some permutation matrix and \(\errorAdjacency\) is defined by~\eqref{eq: errorDef}.
  \end{theorem}
  Finally, we can get a bound on the estimation error of community memberships:
  \begin{theorem}
  \label{theorem: nodeMatrixBound}
    Let us consider the model~\eqref{eq: nmf} and let the Condition~\ref{cond: indentifiability} holds. Let also \(\|\adjacencyMatrix - \probMatrix\| \le \frac{1}{2} \lambda_{\nclusters}(\probMatrix)\). Then SPOC algorithm outputs matrix \(\nodeCommunityMatrixEstimate\) such that it holds
    \begin{EQA}[c]
      \bigl\|\nodeCommunityMatrixEstimate - \nodeCommunityMatrix \mathrm{\Pi}_{\basisMatrix}^{\T}\bigr\|_{F}
      \leq
      C \nclusters^{1/2} \condNumber^{3}(\nodeCommunityMatrix) \lambda_{max}^2(\nodeCommunityMatrix) \, \errorAdjacency + C \nclusters^{1/2} \condNumber(\nodeCommunityMatrix) \lambda_{max}(\nodeCommunityMatrix) \frac{\|\adjacencyMatrix - \probMatrix\|}{\lambda_{\nclusters}(\probMatrix)},
    \end{EQA}
    where \(\lambda_{max}(\nodeCommunityMatrix)\) is the maximum singular value of matrix \(\nodeCommunityMatrix\) and \(\condNumber(\nodeCommunityMatrix)\) is the condition number of matrix \(\nodeCommunityMatrix\), \(\mathrm{\Pi}_{\basisMatrix}\) is a permutation matrix and \(\errorAdjacency\) is defined by~\eqref{eq: errorDef}.
  \end{theorem}
  The bounds of the Theorems~\ref{theorem: communityMatrixBound} and~\ref{theorem: nodeMatrixBound} depend on the properties of matrices \(\nodeCommunityMatrix\) and \(\probMatrix\), which can be further quantified for the particular random graph models.

\section{Tools}
\label{sec: tools}
  This section collects some general statements which are useful for our analysis. We start by the following important lemma which is a variant of Davis-Kahan theorem.
  \begin{lemma}[Lemma 5.1 of~\cite{Lei2015}]
  \label{lemma: eigenvectorConsistency}
    Assume that \(\probMatrix \in \RR^{\nsize \t \nsize}\) is a rank \(\nclusters\) symmetric matrix with smallest nonzero singular value \(\lambda_{\nclusters}(\probMatrix)\). Let \(\adjacencyMatrix\) be any symmetric matrix and \(\adjacencyEigenvectors, \probEigenvectors \in \RR^{\nsize \t \nclusters}\) be the \(\nclusters\) leading eigenvectors of \(\adjacencyMatrix\) and \(\probMatrix\), respectively. Then there exists a \(\nclusters \t \nclusters\) orthogonal matrix \(\orthMatrix_{\probMatrix}\) such that
    \begin{EQA}[c]
      \|\adjacencyEigenvectors - \probEigenvectors \orthMatrix_{\probMatrix}\|_{F} \le \frac{2 \sqrt{2 \nclusters} \|\adjacencyMatrix - \probMatrix\|}{\lambda_{\nclusters}(\probMatrix)}.
    \end{EQA}
  \end{lemma}
  Based on this result it is quite straightforward to get the following bounds for the matrix of eigenvalues.
  \begin{corollary}
  \label{corollary: eigenvalues}
    Let us assume that the conditions of Lemma~\ref{lemma: eigenvectorConsistency} hold. Let \(\adjacencyEigenvalues, \probEigenvalues\) be diagonal \(\nclusters \times \nclusters\)-matrices with \(\nclusters\) largest in absolute value eigenvalues of \(\adjacencyMatrix\) and \(\probMatrix\) respectively on the diagonal. Then it holds
    \begin{EQA}[c]
      \|\adjacencyEigenvalues - \orthMatrix_{\probMatrix}^{\T} \probEigenvalues \orthMatrix_{\probMatrix}\|
      \le
      \Bigl(2 \sqrt{2 \nclusters} \frac{\|\adjacencyMatrix\| + \|\probMatrix\|}{\lambda_{\nclusters}(\probMatrix)} + 1\Bigr) ~ \|\adjacencyMatrix - \probMatrix\|
    \end{EQA}
    and
    \begin{EQA}[c]
      \|\adjacencyEigenvalues^{-1} - \orthMatrix_{\probMatrix}^{\T} \probEigenvalues^{-1} \orthMatrix_{\probMatrix}\|
      \le
      \Bigl(2 \sqrt{2 \nclusters} \frac{\|\adjacencyMatrix\| + \|\probMatrix\|}{\lambda_{\nclusters}(\probMatrix)} + 1\Bigr) \frac{\|\adjacencyMatrix - \probMatrix\|}{\lambda_{\nclusters}(\adjacencyMatrix) \cdot \lambda_{\nclusters}(\probMatrix)},
    \end{EQA}
    where the orthogonal matrix \(\orthMatrix_{\probMatrix}\) is the same as in Lemma~\ref{lemma: eigenvectorConsistency}.
  \end{corollary}

  \begin{proof}
    We start by noting that
    \begin{EQA}[c]
      \|\adjacencyEigenvectors \adjacencyEigenvalues \adjacencyEigenvectors^{\T} - \probEigenvectors \probEigenvalues \probEigenvectors^{\T}\| \le \|\adjacencyMatrix - \probMatrix\|
    \end{EQA}
    and further
    \begin{EQA}
      &&\|\adjacencyEigenvectors \adjacencyEigenvalues \adjacencyEigenvectors^{\T} - \probEigenvectors \orthMatrix_{\probMatrix} \orthMatrix_{\probMatrix}^{\T} \probEigenvalues \orthMatrix_{\probMatrix} \orthMatrix_{\probMatrix}^{\T} \probEigenvectors^{\T}\|
      \ge
      \|\probEigenvectors \orthMatrix_{\probMatrix} (\adjacencyEigenvalues - \orthMatrix_{\probMatrix}^{\T} \probEigenvalues \orthMatrix_{\probMatrix}) \adjacencyEigenvectors^{\T}\|
      \\
      &-&
      \|(\adjacencyEigenvectors - \probEigenvectors \orthMatrix_{\probMatrix}) \adjacencyEigenvalues \adjacencyEigenvectors^{\T}\|
      -
      \|\probEigenvectors \probEigenvalues \orthMatrix_{\probMatrix} (\adjacencyEigenvectors - \probEigenvectors \orthMatrix_{\probMatrix})^{\T}\|.
    \end{EQA}
    Then
    \begin{EQA}
      && \|\adjacencyEigenvalues - \orthMatrix_{\probMatrix}^{\T} \probEigenvalues \orthMatrix_{\probMatrix}\|
      \le
      \|\adjacencyMatrix - \probMatrix\|
      +
      \|(\adjacencyEigenvectors - \probEigenvectors \orthMatrix_{\probMatrix}) \adjacencyEigenvalues \adjacencyEigenvectors^{\T}\|
      +
      \|\probEigenvectors \probEigenvalues \orthMatrix_{\probMatrix} (\adjacencyEigenvectors - \probEigenvectors \orthMatrix_{\probMatrix})^{\T}\|
      \\
      &\le&
      \|\adjacencyMatrix - \probMatrix\| + (\|\adjacencyMatrix\| + \|\probMatrix\|) \|\adjacencyEigenvectors - \probEigenvectors \orthMatrix_{\probMatrix}\|
      \le
      \Bigl(2 \sqrt{2 \nclusters} \frac{\|\adjacencyMatrix\| + \|\probMatrix\|}{\lambda_{\nclusters}(\probMatrix)} + 1\Bigr) ~ \|\adjacencyMatrix - \probMatrix\|. 
    \end{EQA}
    where the last inequality is due to Lemma~\ref{lemma: eigenvectorConsistency}.
    Now we can bound
    \begin{EQA}
      &&\|\adjacencyEigenvalues^{-1} - \orthMatrix_{\probMatrix}^{\T} \probEigenvalues^{-1} \orthMatrix_{\probMatrix}\| 
      =
      \|\adjacencyEigenvalues^{-1}(\orthMatrix_{\probMatrix}^{\T} \probEigenvalues \orthMatrix_{\probMatrix} - \adjacencyEigenvalues) \orthMatrix_{\probMatrix}^{\T} \probEigenvalues^{-1} \orthMatrix_{\probMatrix}\|
      \\
      &\le&
      \|\adjacencyEigenvalues^{-1}\| \cdot \|\adjacencyEigenvalues - \orthMatrix_{\probMatrix}^{\T} \probEigenvalues \orthMatrix_{\probMatrix}\| \cdot \|\probEigenvalues^{-1}\|
      \le
      \Bigl(2 \sqrt{2 \nclusters} \frac{\|\adjacencyMatrix\| + \|\probMatrix\|}{\lambda_{\nclusters}(\probMatrix)} + 1\Bigr) \frac{\|\adjacencyMatrix - \probMatrix\|}{\lambda_{\nclusters}(\adjacencyMatrix) \cdot \lambda_{\nclusters}(\probMatrix)},
    \end{EQA}
    where \(\lambda_{\nclusters}(\adjacencyMatrix), \lambda_{\nclusters}(\probMatrix)\) are the \(\nclusters\)-th largest in absolute value eigenvalues of matrices \(\adjacencyMatrix\) and \(\probMatrix\) respectively.
  \end{proof}

  The following result gives a tight bound on spectral norm for the centered symmetric matrix of independent Bernoulli variables.
  \begin{lemma}[Theorem 5.2 of~\cite{Lei2015}]
  \label{lemma: adjacencyConcentration}
    Let \(\adjacencyMatrix\) be the adjacency matrix of a random graph on \(\nsize\) nodes in which edges occur independently. Set \(\EE[\adjacencyMatrix] = \probMatrix = (p_{ij})_{ i, j = 1, \dots, \nsize}\) and assume that \(\nsize \max_{ij} p_{ij} \leq d\) for \(d \geq c_{0} \log \nsize\) and \(c_{0} > 0\). Then, for any \(r > 0\) there exists a constant \(C = C(r, c_{0})\) such that
    \begin{EQA}[c]
      \|\adjacencyMatrix - \probMatrix\| \leq C \sqrt{d}
    \end{EQA}
    with probability at least \(1 - \nsize^{-r}\).
  \end{lemma}
  Also we want to remind the matrix Chernoff inequality.
  \begin{theorem}[Matrix Chernoff, Theorem~1.1 of \cite{Tropp2012}]
  \label{theorem: chernoff}
    Consider a finite sequence \({X_k}\) of independent, random, self-adjoint matrices with dimension \(\nclusters\). Assume that each random matrix satisfies
    \begin{EQA}[c]
      X_{k} \geq 0 ~~ \text{and} ~~ \lambda_{max}(X_k) \le R ~ \text{almost surely}.
    \end{EQA}
    Define
    \begin{EQA}[c]
      \mu_{min} = \lambda_{min}\Bigl(\sum_k \EE X_k\Bigr) ~~ \text{and} ~~ \mu_{max} = \lambda_{max}\Bigl(\sum_k \EE X_k\Bigr).
    \end{EQA}
    Then
      \begin{EQA}
        && \PP\Bigl\{\lambda_{min}\Bigl(\sum_k X_k\Bigr) \leq (1 - \delta) \mu_{min}\Bigr\} \leq \nclusters \left[\frac{\ex^{-\delta}}{(1 - \delta)^{1 - \delta}}\right]^{\frac{\mu_{min}}{R}} ~ \text{for} ~ \delta \in [0, 1], ~ \text{and}
        \\
        && \PP\Bigl\{\lambda_{max}\Bigl(\sum_k X_k\Bigr) \geq (1 + \delta) \mu_{max}\Bigr\} \leq \nclusters \left[\frac{\ex^{\delta}}{(1 + \delta)^{1 + \delta}}\right]^{\frac{\mu_{max}}{R}} ~ \text{for} ~ \delta \geq 0.
      \end{EQA}
  \end{theorem}
  The following corollary is particularly useful for our analysis.
  \begin{corollary}[\cite{Tropp2012}]
  \label{corollary: chernoffSimple}
    Under the conditions of Theorem~\ref{theorem: chernoff} it holds
    \begin{EQA}
        && \PP\Bigl\{\lambda_{min}\Bigl(\sum_k X_k\Bigr) \leq t \mu_{min}\Bigr\} \leq \nclusters \ex^{-(1 - t)^2 \mu_{min} / 2R} ~ \text{for} ~ t \in [0, 1], ~ \text{and}
        \\
        && \PP\Bigl\{\lambda_{max}\Bigl(\sum_k X_k\Bigr) \geq t \mu_{max}\Bigr\} \leq \nclusters \left[\frac{\ex}{t}\right]^{t\mu_{max} / R} ~ \text{for} ~ t \geq \ex.
      \end{EQA}
  \end{corollary}
  We finish the section by proving the following lemma.
  \begin{lemma}
  \label{lemma: matrixSquare}
    Let for two \(\nclusters \t \nclusters\) full-rank matrices \(U_1\) and \(U_2\) it holds that \(\|U_1 - U_2\|_{F} \le \eps\). Then
    \begin{EQA}[c]
      \bigl\|U_1 U_1^{\T} - U_2 U_2^{\T}\bigr\|_{F} \le (\|U_1\| + \|U_2\|) \eps
    \end{EQA}
    and
    \begin{EQA}[c]
      \bigl\|\bigl(U_1 U_1^{\T}\bigr)^{-1} - \bigl(U_2 U_2^{\T}\bigr)^{-1}\bigr\|_{F} \le \frac{\bigl\|U_1\bigr\| + \bigl\|U_2\bigr\|}{\lambda_{min}^{2}(U_1) \lambda_{min}^2(U_2)} \eps.
    \end{EQA}
  \end{lemma}
  \begin{proof}
    The first result follows from the following sequence of inequalities:
    \begin{EQA}
      && \bigl\|U_1 U_1^{\T} - U_2 U_2^{\T}\bigr\|_{F} = \bigl\|U_1 (U_1 - U_2)^{\T} - (U_2 - U_1) U_2^{\T}\bigr\|_{F} \le \bigl\|U_1 (U_1 - U_2)^{\T}\bigr\|_{F} + \bigl\|(U_2 - U_1) U_2^{\T}\bigr\|_{F}
      \\
      &\le&
      \bigl(\bigl\|U_1\bigr\| + \bigl\|U_2\bigr\|\bigr) \bigl\|U_1 - U_2\bigr\|_{F} \le \bigl(\bigl\|U_1\bigr\| + \bigl\|U_2\bigr\|\bigr) \eps.
    \end{EQA}
    The second result holds due to
    \begin{EQA}
      && \bigl\|\bigl(U_1 U_1^{\T}\bigr)^{-1} - \bigl(U_2 U_2^{\T}\bigr)^{-1}\bigr\|_{F} = \bigl\|\bigl(U_1 U_1^{\T}\bigr)^{-1} \bigl(U_2 U_2^{\T} - U_1 U_1^{\T}\bigr) \bigl(U_2 U_2^{\T}\bigr)^{-1}\bigr\|_{F}
      \\
      &\le&
      \bigl\|\bigl(U_1 U_1^{\T}\bigr)^{-1}\bigr\| \cdot \bigl\|\bigl(U_2 U_2^{\T}\bigr)^{-1}\bigr\| \cdot \bigl\|U_2 U_2^{\T} - U_1 U_1^{\T}\bigr\|_{F} \le \frac{\bigl\|U_1\bigr\| + \bigl\|U_2\bigr\|}{\lambda_{min}^{2}(U_1) \lambda_{min}^2(U_2)} \eps.
    \end{EQA}
  \end{proof}

\section{Proofs}
\label{sec: proofs}
  This section collects the proofs of the main results.

\subsection{Proof of Theorem~\ref{theorem: identifiability}}
  We start by noting that if model~\eqref{eq: mmsbDef} satisfies the Condition~\ref{cond: indentifiability}, then \(rank(\probMatrix) = rank(\communityMatrix) = \nclusters\), which means that the parameter \(\nclusters\) is identifiable.

  Let us assume that \(\nodeCommunityMatrix, \nodeCommunityMatrix' \in \nodeCommunitySet\) and \(\communityMatrix, \communityMatrix'\) are invertible matrices such that \(\probMatrix = \nodeCommunityMatrix \communityMatrix \nodeCommunityMatrix^{\T} = \nodeCommunityMatrix' \communityMatrix' {\nodeCommunityMatrix'}^{\T}\). We show that there exists some permutation \(\sigma\) such that \(\nodeCommunityMatrix = \nodeCommunityMatrix' \mathrm{\Pi}_{\sigma}\) and \(\communityMatrix = \mathrm{\Pi}_{\sigma^{-1}} \communityMatrix' \mathrm{\Pi}_{\sigma^{-1}}^{\T}\).

  Let \(\probEigenvectors\) be a matrix containing \(\nclusters\) independent normalized eigenvectors of \(\probMatrix\) associated to non-zero eigenvalues. The columns of \(\probEigenvectors\) form a basis and there exist invertible matrices \(X, X'\) such that \(\probEigenvectors = \nodeCommunityMatrix X = \nodeCommunityMatrix' X'\).

  We further note, that for all \(k = 1, \dots, \nclusters\) there exists some \(i_{k}\) such that \(\theta_{i_{k}, j} = \delta_{j, k}\) for \(j = 1, \dots, \nclusters\). It means, that \(k\)-th row of \(X\) can be represented as a weighted sum of rows in \(X'\):
  \begin{EQA}[c]
  \label{eq: ident_representation}
    X_{k} = \sum_{l = 1}^{\nclusters} \theta_{i_{k}, l}^{'} X_{l}^{'}.
  \end{EQA}
  The same can be done for any row in \(X'\). If we substitute each \(X_{l}^{'}\) by the corresponding convex combination then we obtain
  \begin{EQA}[c]
    X_{k} = \sum_{m = 1}^{\nclusters} a_{m} X_{m},
  \end{EQA}
  where
  \begin{EQA}[c]
    a_{m} = \sum_{l = 1}^{\nclusters} \theta_{i_{k}, l}^{'} \theta_{i_{l}^{'}, m}.
  \end{EQA}
  Due to the fact, the matrix \(X\) is full rank we conclude that \(a_{m} = \delta_{m, k}\). Further \(a_{k} = 1\) is equivalent to the fact that \(\theta_{i_{l}^{'}, m} = 1\) for that values of \(l\) which correspond to \(\theta_{i_{k}, l}^{'} > 0\). At least one such \(l\) exists, which means that
  \begin{EQA}[c]
    X_{l}^{'} = X_{m}.
  \end{EQA}
  So we can find pairwise correspondence between rows of \(X\) and \(X'\) which is necessary a perfect matching as both matrices are full rank. We can conclude, that \(X' = \mathrm{\Pi}_{\sigma} X\) for some permutation \(\sigma\). We deduce that \(\nodeCommunityMatrix \communityMatrix {\nodeCommunityMatrix}^{\T} =  \nodeCommunityMatrix \mathrm{\Pi}_{\sigma^{-1}} \communityMatrix' \mathrm{\Pi}_{\sigma^{-1}}^{\T} {\nodeCommunityMatrix}^{\T}\) and \(\communityMatrix = \mathrm{\Pi}_{\sigma^{-1}} \communityMatrix' \mathrm{\Pi}_{\sigma^{-1}}^{\T}\) as mapping \(\nodeCommunityMatrix\) is injective.

\subsection{Proof of Lemma~\ref{lemma: rowFactorBound}}
  We start by upper bounding
  \begin{EQA}
    && \|\ev_{i}^{\T}(\adjacencyEigenvectors - \probEigenvectors \orthMatrix_{\probMatrix})\|_{F}
    =
    \|\ev_{i}^{\T}(\adjacencyMatrix \adjacencyEigenvectors \adjacencyEigenvalues^{-1} - \probMatrix \probEigenvectors  \probEigenvalues^{-1} \orthMatrix_{\probMatrix})\|_{F}
    \\
    &=&
    \|\ev_{i}^{\T} \adjacencyMatrix \adjacencyEigenvectors (\adjacencyEigenvalues^{-1} - \orthMatrix_{\probMatrix}^{\T} \probEigenvalues^{-1} \orthMatrix_{\probMatrix}) + \ev_{i}^{\T} \adjacencyMatrix (\adjacencyEigenvectors - \probEigenvectors \orthMatrix_{\probMatrix}) \orthMatrix_{\probMatrix}^{\T} \probEigenvalues^{-1} \orthMatrix_{\probMatrix} + \ev_{i}^{\T}(\adjacencyMatrix - \probMatrix) \probEigenvectors \probEigenvalues^{-1} \orthMatrix_{\probMatrix}\|_{F}
    \\
    &\le&
    \|\ev_{i}^{\T} \adjacencyMatrix \adjacencyEigenvectors (\adjacencyEigenvalues^{-1} - \orthMatrix_{\probMatrix}^{\T} \probEigenvalues^{-1} \orthMatrix_{\probMatrix})\|_{F} + \|\ev_{i}^{\T} \adjacencyMatrix (\adjacencyEigenvectors - \probEigenvectors \orthMatrix_{\probMatrix}) \orthMatrix_{\probMatrix}^{\T} \probEigenvalues^{-1} \orthMatrix_{\probMatrix}\|_{F} + \|\ev_{i}^{\T}(\adjacencyMatrix - \probMatrix) \probEigenvectors \probEigenvalues^{-1} \orthMatrix_{\probMatrix}\|_{F}
    \\
    &=&
    I_1 + I_2 + I_3.
  \end{EQA}
  Let us bound these three terms separately. For the first term we proceed as
  \begin{EQA}
    I_1
    &=&
    \|\ev_{i}^{\T} \adjacencyMatrix \adjacencyEigenvectors (\adjacencyEigenvalues^{-1} - \orthMatrix_{\probMatrix}^{\T} \probEigenvalues^{-1} \orthMatrix_{\probMatrix})\|_{F}
    \le
    \|\ev_{i}^{\T} \adjacencyMatrix\|_{F} \cdot \|\adjacencyEigenvectors\| \cdot \|\adjacencyEigenvalues^{-1} - \orthMatrix_{\probMatrix}^{\T} \probEigenvalues^{-1} \orthMatrix_{\probMatrix}\|
    \\
    &\le&
    \Bigl(2 \sqrt{2 \nclusters} \frac{\|\adjacencyMatrix\| + \|\probMatrix\|}{\lambda_{\nclusters}(\probMatrix)} + 1\Bigr) \frac{\|\ev_{i}^{\T} \adjacencyMatrix\|_{F} \cdot \|\adjacencyMatrix - \probMatrix\|}{\lambda_{\nclusters}(\adjacencyMatrix) \cdot \lambda_{\nclusters}(\probMatrix)}
    \le
    20 \nclusters^{1/2} \condNumber(\probMatrix) ~ \frac{\|\ev_{i}^{\T} \adjacencyMatrix\|_{F}\cdot \|\adjacencyMatrix - \probMatrix\|}{\lambda_{\nclusters}^{2}(\probMatrix)},
  \end{EQA}
  where the last two inequalities are due to Corollary~\ref{corollary: eigenvalues} and the condition \(\|\adjacencyMatrix - \probMatrix\| \le \frac{1}{2} \lambda_{\nclusters}(\probMatrix)\).
  The other two terms can be bounded using the bounds for the norm of matrix product
  \begin{EQA}
    && I_2
    =
    \|\ev_{i}^{\T} \adjacencyMatrix (\adjacencyEigenvectors - \probEigenvectors \orthMatrix_{\probMatrix}) \orthMatrix_{\probMatrix}^{\T} \probEigenvalues^{-1} \orthMatrix_{\probMatrix}\|_{F}
    \le
    \|\ev_{i}^{\T} \adjacencyMatrix\|_{F} \cdot \|\adjacencyEigenvectors - \probEigenvectors \orthMatrix_{\probMatrix}\| \cdot \|\probEigenvalues^{-1}\|
    =
    \frac{\|\ev_{i}^{\T} \adjacencyMatrix\|_{F} \cdot \|\adjacencyEigenvectors - \probEigenvectors \orthMatrix_{\probMatrix}\|}{\lambda_{\nclusters}(\probMatrix)}
    \\
    &\le&
    2 \sqrt{2 \nclusters} \frac{\|\ev_{i}^{\T} \adjacencyMatrix\|_{F} \cdot \|\adjacencyMatrix - \probMatrix\|}{\lambda_{\nclusters}^2(\probMatrix)}
    \le
    3 \nclusters^{1/2} \frac{\|\ev_{i}^{\T} \adjacencyMatrix\|_{F} \cdot \|\adjacencyMatrix - \probMatrix\|}{\lambda_{\nclusters}^2(\probMatrix)}
  \end{EQA}
  and
  \begin{EQA}
    && I_3
    =
    \|\ev_{i}^{\T}(\adjacencyMatrix - \probMatrix) \probEigenvectors \probEigenvalues^{-1} \orthMatrix_{\probMatrix}\|_{F}
    \le
    \|\ev_{i}^{\T}(\adjacencyMatrix - \probMatrix) \probEigenvectors\|_{F} \cdot \|\probEigenvalues^{-1}\|
    =
    \frac{\|\ev_{i}^{\T}(\adjacencyMatrix - \probMatrix) \probEigenvectors\|_{F}}{\lambda_{\nclusters}(\probMatrix)}.
  \end{EQA}
  Combination of these bounds gives the desired result.

\subsection{Proof of Theorem~\ref{theorem: communityMatrixBound}}
  We start by the following sequence of inequalities.
  \begin{EQA}
    && \bigl\|\basisMatrixEstimate \adjacencyEigenvalues \basisMatrixEstimate^{\T} - \mathrm{\Pi}_{\basisMatrix} \basisMatrix \probEigenvalues \basisMatrix^{\T} \mathrm{\Pi}_{\basisMatrix}^{\T}\bigr\|_{F}
    \le
    \bigl\|(\basisMatrixEstimate - \mathrm{\Pi}_{\basisMatrix} \basisMatrix \orthMatrix_{\probMatrix}) \orthMatrix_{\probMatrix}^{\T} \probEigenvalues \basisMatrix^{\T} \mathrm{\Pi}_{\basisMatrix}^{\T}\bigr\|_{F}
    +
    \bigl\|\basisMatrixEstimate (\adjacencyEigenvalues -  \orthMatrix_{\probMatrix}^{\T}\probEigenvalues \orthMatrix_{\probMatrix}) \orthMatrix_{\probMatrix}^{\T} \basisMatrix^{\T} \mathrm{\Pi}_{\basisMatrix}^{\T}\bigr\|_{F}
    \\
    &+&
    \bigl\|\basisMatrixEstimate \adjacencyEigenvalues (\basisMatrixEstimate - \mathrm{\Pi}_{\basisMatrix} \basisMatrix \orthMatrix_{\probMatrix})^{\T}\bigr\|_{F} = I_1 + I_2 + I_3.
  \end{EQA}
  We bound three terms separately:
  \begin{EQA}
    && I_1
    =
    \bigl\|(\basisMatrixEstimate - \mathrm{\Pi}_{\basisMatrix} \basisMatrix \orthMatrix_{\probMatrix}) \orthMatrix_{\probMatrix}^{\T} \probEigenvalues \basisMatrix^{\T} \mathrm{\Pi}_{\basisMatrix}^{\T}\bigr\|_{F}
    \le
    \bigl\|\basisMatrixEstimate - \mathrm{\Pi}_{\basisMatrix} \basisMatrix \orthMatrix_{\probMatrix}\bigr\|_{F}
    \cdot \|\probEigenvalues\| \cdot \|\basisMatrix\|
    \\
    &\le&
    C_0 \nclusters^{1/2} \condNumber(\basisMatrix) \|\probMatrix\| \cdot \|\basisMatrix\| \errorAdjacency.
  \end{EQA}
  For the second term we get
  \begin{EQA}
    && I_2
    =
    \bigl\|\basisMatrixEstimate (\adjacencyEigenvalues -  \orthMatrix_{\probMatrix}^{\T}\probEigenvalues \orthMatrix_{\probMatrix}) \orthMatrix_{\probMatrix}^{\T} \basisMatrix^{\T} \mathrm{\Pi}_{\basisMatrix}^{\T}\bigr\|_{F}
    \le
    \|\basisMatrixEstimate\| \cdot \|\adjacencyEigenvalues -  \orthMatrix_{\probMatrix}^{\T}\probEigenvalues \orthMatrix_{\probMatrix}\| \cdot \|\basisMatrix\|
    \\
    &\le&
    8 \nclusters^{1/2} \condNumber(\probMatrix) \|\basisMatrix\|^2 \cdot \|\adjacencyMatrix - \probMatrix\|.
  \end{EQA}
  Finally, by analogy with the first term we obtain for the third term \(I_3 = \bigl\|\basisMatrixEstimate \adjacencyEigenvalues (\basisMatrixEstimate - \mathrm{\Pi}_{\basisMatrix} \basisMatrix \orthMatrix_{\probMatrix})^{\T}\bigr\|_{F} \le 4 C_0 \nclusters^{1/2} \condNumber(\basisMatrix) \|\probMatrix\| \cdot \|\basisMatrix\| \errorAdjacency\). The combination of the obtained bounds for \(I_1, I_2\) and \(I_3\) gives the final result.

\subsection{Proof of Theorem~\ref{theorem: nodeMatrixBound}}
  We remind that
  \begin{EQA}[c]
    \nodeCommunityMatrixEstimate = \adjacencyEigenvectors \basisMatrixEstimate^{\T} \bigl(\basisMatrixEstimate \basisMatrixEstimate^{\T}\bigr)^{-1}
  \end{EQA}
  and note
  \begin{EQA}
    &&\nodeCommunityMatrix \mathrm{\Pi}_{\basisMatrix}^{\T}
    =
    \probEigenvectors \basisMatrix^{\T} \bigl(\basisMatrix \basisMatrix^{\T}\bigr)^{-1} \mathrm{\Pi}_{\basisMatrix}^{\T}
    =
    \probEigenvectors \orthMatrix_{\probMatrix} \orthMatrix_{\probMatrix}^{\T} \basisMatrix^{\T} \bigl(\mathrm{\Pi}_{\basisMatrix}^{\T} \mathrm{\Pi}_{\basisMatrix} \basisMatrix \basisMatrix^{\T} \mathrm{\Pi}_{\basisMatrix}^{\T} \mathrm{\Pi}_{\basisMatrix}\bigr)^{-1} \mathrm{\Pi}_{\basisMatrix}^{\T}
    \\
    &=&
    \probEigenvectors \orthMatrix_{\probMatrix} \bigr[\mathrm{\Pi}_{\basisMatrix} \basisMatrix \orthMatrix_{\probMatrix}\bigl]^{\T} \bigl(\mathrm{\Pi}_{\basisMatrix} \basisMatrix \basisMatrix^{\T} \mathrm{\Pi}_{\basisMatrix}^{\T}\bigr)^{-1}.
  \end{EQA}
  Let us bound an error of approximation
  \begin{EQA}
    && \bigl\|\nodeCommunityMatrixEstimate - \nodeCommunityMatrix \mathrm{\Pi}_{\basisMatrix}^{\T}\bigr\|_{F}
    =
    \bigl\|\adjacencyEigenvectors \basisMatrixEstimate^{\T} \bigl(\basisMatrixEstimate \basisMatrixEstimate^{\T}\bigr)^{-1} - \probEigenvectors \orthMatrix_{\probMatrix} \bigr[\mathrm{\Pi}_{\basisMatrix} \basisMatrix \orthMatrix_{\probMatrix}\bigl]^{\T} \bigl(\mathrm{\Pi}_{\basisMatrix} \basisMatrix \basisMatrix^{\T} \mathrm{\Pi}_{\basisMatrix}^{\T}\bigr)^{-1}\bigr\|_{F}
    \\
    &\le&
    \bigl\|\adjacencyEigenvectors \basisMatrixEstimate^{\T} \bigl[\bigl(\basisMatrixEstimate \basisMatrixEstimate^{\T}\bigr)^{-1} - \bigl(\mathrm{\Pi}_{\basisMatrix} \basisMatrix \basisMatrix^{\T} \mathrm{\Pi}_{\basisMatrix}^{\T}\bigr)^{-1}\bigr]\bigr\|_{F}
    +
    \bigl\|\adjacencyEigenvectors \bigl[\basisMatrixEstimate - \mathrm{\Pi}_{\basisMatrix} \basisMatrix \orthMatrix_{\probMatrix}\bigr]^{\T} \bigl(\mathrm{\Pi}_{\basisMatrix} \basisMatrix \basisMatrix^{\T} \mathrm{\Pi}_{\basisMatrix}^{\T}\bigr)^{-1}\bigr\|_{F}
    \\
    &+&
    \bigl\|\bigl[\adjacencyEigenvectors - \probEigenvectors \orthMatrix_{\probMatrix}\bigr] \bigl[\mathrm{\Pi}_{\basisMatrix} \basisMatrix \orthMatrix_{\probMatrix}\bigr]^{\T} \bigl(\mathrm{\Pi}_{\basisMatrix} \basisMatrix \basisMatrix^{\T} \mathrm{\Pi}_{\basisMatrix}^{\T}\bigr)^{-1}\bigr\|_{F}
    =
    I_{1} + I_2 + I_3.
  \end{EQA}
  We proceed by bounding each summand separately denoting by \(C > 0\) some sufficiently large constant:
  \begin{EQA}
    && I_1 = \bigl\|\adjacencyEigenvectors \basisMatrixEstimate^{\T} \bigr[\bigl(\basisMatrixEstimate \basisMatrixEstimate^{\T}\bigr)^{-1} - \bigl(\mathrm{\Pi}_{\basisMatrix} \basisMatrix \basisMatrix^{\T} \mathrm{\Pi}_{\basisMatrix}^{\T}\bigr)^{-1}\bigr]\bigr\|_{F}
    \le
    \bigl\|\adjacencyEigenvectors\bigr\| \cdot \bigl\|\basisMatrixEstimate\bigr\| \cdot \bigl\|\bigl(\basisMatrixEstimate \basisMatrixEstimate^{\T}\bigr)^{-1} - \bigl(\mathrm{\Pi}_{\basisMatrix} \basisMatrix \basisMatrix^{\T} \mathrm{\Pi}_{\basisMatrix}^{\T}\bigr)^{-1}\bigr\|_{F}
    \\
    &\le&
    \bigl\|\adjacencyEigenvectors\bigr\| \cdot \bigl(\bigl\|\basisMatrix\bigr\| + \bigl\|\basisMatrixEstimate - \mathrm{\Pi}_{\basisMatrix} \basisMatrix \orthMatrix_{\probMatrix}\bigr\|\bigr) \cdot \bigl\|\bigl(\basisMatrixEstimate \basisMatrixEstimate^{\T}\bigr)^{-1} - \bigl(\mathrm{\Pi}_{\basisMatrix} \basisMatrix \basisMatrix^{\T} \mathrm{\Pi}_{\basisMatrix}^{\T}\bigr)^{-1}\bigr\|_{F}
    \\
    &\le&
    2 \bigl\|\adjacencyEigenvectors\bigr\| \cdot \bigl\|\basisMatrix\bigr\| \cdot \bigl\|\bigl(\basisMatrixEstimate \basisMatrixEstimate^{\T}\bigr)^{-1} - \bigl(\mathrm{\Pi}_{\basisMatrix} \basisMatrix \basisMatrix^{\T} \mathrm{\Pi}_{\basisMatrix}^{\T}\bigr)^{-1}\bigr\|_{F}
    \le
    2 \lambda_{max}(\basisMatrix) 6 C_0 \nclusters^{1/2} \frac{\condNumber^{2}(\basisMatrix)}{\lambda_{min}^{3}(\basisMatrix)} \errorAdjacency
    \\
    &=&
    12 C_0 \nclusters^{1/2} \frac{\condNumber^{3}(\basisMatrix)}{\lambda_{min}^{2}(\basisMatrix)} \errorAdjacency.
  \end{EQA}
  Here we use the bound
  \begin{EQA}[c]
    \bigl\|\bigl(\basisMatrixEstimate \basisMatrixEstimate^{\T}\bigr)^{-1} - \bigl(\mathrm{\Pi}_{\basisMatrix} \basisMatrix \basisMatrix^{\T} \mathrm{\Pi}_{\basisMatrix}^{\T}\bigr)^{-1}\bigr\|_{F}
    \le
    6 C_0 \nclusters^{1/2} \frac{\condNumber^2(\basisMatrix)}{\lambda_{min}^3(\basisMatrix)} \errorAdjacency,
  \end{EQA}
  which follows from Lemma~\ref{lemma: matrixSquare}.

  We continue by bounding
  \begin{EQA}
    && I_2 = \bigl\|\adjacencyEigenvectors \bigl[\basisMatrixEstimate - \mathrm{\Pi}_{\basisMatrix} \basisMatrix \orthMatrix_{\probMatrix}\bigr]^{\T} \bigl(\mathrm{\Pi}_{\basisMatrix} \basisMatrix \basisMatrix^{\T} \mathrm{\Pi}_{\basisMatrix}^{\T}\bigr)^{-1}\bigr\|_{F}
    \le
    \bigl\|\adjacencyEigenvectors\bigr\| \bigl\|\basisMatrixEstimate - \mathrm{\Pi}_{\basisMatrix} \basisMatrix \orthMatrix_{\probMatrix}\bigr\|_{F} \bigl\|\bigl(\basisMatrix \basisMatrix^{\T}\bigr)^{-1}\bigr\|
    \\
    &\le&
    \bigl\|\adjacencyEigenvectors\bigr\| \cdot \bigl\|\basisMatrixEstimate - \mathrm{\Pi}_{\basisMatrix} \basisMatrix \orthMatrix_{\probMatrix}\bigr\|_{F} \cdot \bigl\|\bigl(\basisMatrix \basisMatrix^{\T}\bigr)^{-1}\bigr\|
    \le
    2 C_0 \nclusters^{1/2} \condNumber(\basisMatrix) \errorAdjacency \frac{1}{\lambda_{min}^{2}(\basisMatrix)}
    \\
    &=&
    2 C_0 \nclusters^{1/2} \frac{\condNumber(\basisMatrix) }{\lambda_{min}^{2}(\basisMatrix)} \errorAdjacency.
  \end{EQA}
  For the last term we get
  \begin{EQA}
    && I_3 = \bigl\|\bigl[\adjacencyEigenvectors - \probEigenvectors \orthMatrix_{\probMatrix}\bigr] \bigl[\mathrm{\Pi}_{\basisMatrix} \basisMatrix \orthMatrix_{\probMatrix}\bigr]^{\T} \bigl(\mathrm{\Pi}_{\basisMatrix} \basisMatrix \basisMatrix^{\T} \mathrm{\Pi}_{\basisMatrix}^{\T}\bigr)^{-1}\bigr\|_{F}
    \\
    &\le&
    \bigl\|\adjacencyEigenvectors - \probEigenvectors \orthMatrix_{\probMatrix}\bigr\|_{F} \cdot \bigl\|\basisMatrix\bigr\| \cdot \bigl\|\bigl(\basisMatrix \basisMatrix^{\T}\bigr)^{-1}\bigr\|
    \le
    2 \sqrt{2 \nclusters}\frac{\condNumber(\basisMatrix)}{\lambda_{min}(\basisMatrix)} \frac{\|\adjacencyMatrix - \probMatrix\|}{\lambda_{\nclusters}(\probMatrix)}.
  \end{EQA}
  Finally, we can bound
  \begin{EQA}[c]
    \bigl\|\nodeCommunityMatrixEstimate - \nodeCommunityMatrix \mathrm{\Pi}_{\basisMatrix}^{\T}\bigr\|_{F}
    \le
    12 C_0 \nclusters^{1/2} \frac{\condNumber^{3}(\basisMatrix) + \condNumber(\basisMatrix)}{\lambda_{min}^2(\basisMatrix)} \errorAdjacency + 2 \sqrt{2 \nclusters}\frac{\condNumber(\basisMatrix)}{\lambda_{min}(\basisMatrix)} \frac{\|\adjacencyMatrix - \probMatrix\|}{\lambda_{\nclusters}(\probMatrix)}
  \end{EQA}
  and the claimed bound follows in a view of \(\lambda_{min}(\basisMatrix) = 1 / \lambda_{max}(\nodeCommunityMatrix), \lambda_{max}(\basisMatrix) = 1 / \lambda_{\nclusters}(\nodeCommunityMatrix)\) and \(\condNumber(\basisMatrix) = \condNumber(\nodeCommunityMatrix)\).

\subsection{Proof of Theorem~\ref{theorem: mainBound}}
  We start from following simple fact.
  \begin{lemma}
  \label{lemma: probMaxValue}
    Under the model~\eqref{eq: mmsbDef} it holds
    \begin{EQA}[c]
    \label{eq: sparsityDef}
      \max_{i, j} \probMatrix_{i, j} = \sparsityParam = \max_{k, l} \communityMatrix_{k, l}.
    \end{EQA}
  \end{lemma}
  The next lemma deals with singular values of matrices \(\nodeCommunityMatrix\) and \(\probMatrix\).

  \begin{lemma}
  \label{lemma: eigenBoundProb}
    Let's consider the model~\eqref{eq: mmsbDef} and let the Condition~\ref{cond: community memberships} is satisfied. Then there exist such constants \(\bar{c}\) and \(\bar{C}\) depending only on the distribution \(\PP_{\nodeWeights}\) of vectors \(\nodeWeights_{i}\) such that with probability at least \(1 - \ex^{-\nsize}\)
    \begin{EQA}[c]
      \label{eq: eigenBoundNodeCommunity}
      \sqrt{\bar{c} \nsize} \le \lambda_{\nclusters}(\nodeCommunityMatrix) \le \lambda_{max}(\nodeCommunityMatrix) \le \sqrt{\bar{C} \nsize};
      \\
      \bar{c} \lambda_{min}(\communityMatrix) \nsize \leq \lambda_{\nclusters}(\probMatrix) \leq \bar{C} \lambda_{min}(\communityMatrix) \nsize
      \label{eq: eigenBoundMin}
    \end{EQA}
    and
    \begin{EQA}[c]
      \bar{c} \lambda_{max}(\communityMatrix) \nsize \leq \lambda_{max}(\probMatrix) \leq \bar{C} \lambda_{max}(\communityMatrix) \nsize.
      \label{eq: eigenBoundMax}
    \end{EQA}
  \end{lemma}
  We note that inequalities~\eqref{eq: eigenBoundNodeCommunity}, \eqref{eq: eigenBoundMin} and \eqref{eq: eigenBoundMax} can be used as deterministic bounds on the behaviour of eigenvalues of matrices \(\nodeCommunityMatrix\) and \(\probMatrix\) without considering any probabilistic interpretation.

  Let us continue with the proof of main results. We mainly need to bound all the quantities involved in the definition of \(\errorAdjacency\). All the statements below hold with a high probability
  \begin{enumerate}
    \item We start by noting that \(\|\adjacencyMatrix - \probMatrix\| \le C \sqrt{d}\) for some \(d \ge \nsize \sparsityParam \vee c_0 \log \nsize\) with probability at least \(1 - \nsize^{-r}\) due to Lemma~\ref{lemma: adjacencyConcentration}.

    \item Next, \(\max_{i} \|\ev_{i}^{\T} \adjacencyMatrix\|_{F}\) can be bounded by simple sequence of inequalities:
    \begin{EQA}
      && \max_{i} \|\ev_{i}^{\T} \adjacencyMatrix\|_{F}
      \leq
      \max_i \|\ev_{i}^{\T} (\adjacencyMatrix - \probMatrix)\|_{F} + \max_{i} \|\ev_{i}^{\T} \probMatrix\|_{F}
      \leq
      \|\adjacencyMatrix - \probMatrix\| + \max_{i} \|\ev_{i}^{\T} \probMatrix\|_{F}
      \\
      &\le&
      \|\adjacencyMatrix - \probMatrix\| + \sparsityParam \sqrt{\nsize} 
      \leq
      C_0 \sqrt{\sparsityParam \nsize} + \sparsityParam \sqrt{\nsize} \leq C \sqrt{\sparsityParam \nsize},
    \end{EQA}
    where the \(\|\adjacencyMatrix - \probMatrix\|\) is bounded with probability at least \(1 - \nsize^{-r}\) using the result from Lemma~\ref{lemma: adjacencyConcentration}.

    \item Further, \(\max_{i} \|\ev_{i}^{\T}(\adjacencyMatrix - \probMatrix) \probEigenvectors\|_{F}\) can be bounded as
    \begin{EQA}
      && \PP\bigl(\|\ev_{i}^{\T} (\adjacencyMatrix - \probMatrix) \probEigenvectors\|_{F} \geq t\bigr)
      =
      \PP\biggl(\sum_{k = 1}^{\nclusters} \biggl[\sum_{j = 1}^{\nsize} (a_{ij} - p_{ij}) u_{jk}\biggr]^2 \geq t^2\biggr)
      \\
      &\le&
      \PP\biggl(\max_{k} \biggl[\sum_{j = 1}^{\nsize} (a_{ij} - p_{ij}) u_{jk}\biggr]^2 \geq t^2 / \nclusters\biggr)
      \le
      2 \sum_{k = 1}^{\nclusters} \PP\biggl(\sum_{j = 1}^{\nsize} (a_{ij} - p_{ij}) u_{jk} \geq t / \nclusters^{1/2}\biggr)
      \\
      &\le&
      2 \nclusters \exp\bigl(-t^2 / 2 \nclusters\bigr),
    \end{EQA}
    where the last inequality follows from Azuma's inequality and the fact that \(\sum_{j = 1}^{\nsize} u_{j1}^2 = 1\). Now we again apply union bound and get
    \begin{EQA}
      \PP(\max_{i} \|\ev_{i}^{\T} (\adjacencyMatrix - \probMatrix) \probEigenvectors\|_{F} \geq t)
      \le
      \sum_{i = 1}^{\nsize} \PP(\|\ev_{i}^{\T} (\adjacencyMatrix - \probMatrix) \probEigenvectors\|_{F} \geq t)
      \le
      2 \nsize \nclusters \exp\bigl(-t^2 / 2 \nclusters \bigr).
    \end{EQA}
    By taking \(t_r = \sqrt{4 \nclusters \log \frac{\nsize^{1 + r}}{\nclusters}}\) with some \(r > 0\) we achieve that \(\max_{i} \|\ev_{i}^{\T} (\adjacencyMatrix - \probMatrix) \probEigenvectors\|_{F} \le t_r\) with probability at least \(1 - \nsize^{-r}\).
  \end{enumerate}

  Finally, we can bound
  \begin{EQA}
    && \errorAdjacency
    =
    \max_{i}
    \biggl[23 \nclusters^{1/2} \condNumber(\probMatrix) ~ \frac{\|\ev_{i}^{\T} \adjacencyMatrix\|_{F}\cdot \|\adjacencyMatrix - \probMatrix\|}{\lambda_{\nclusters}^{2}(\probMatrix)}
    +
    \frac{\|\ev_{i}^{\T}(\adjacencyMatrix - \probMatrix) \probEigenvectors\|_{F}}{\lambda_{\nclusters}(\probMatrix)}\biggr]
    \\
    &\le&
    C \nclusters^{1/2} \frac{\sqrt{\sparsityParam \nsize} \cdot \sqrt{\sparsityParam \nsize}}{(\sparsityParam \nsize)^{2}} + C \frac{\nclusters^{1/2} \sqrt{\log \nsize}}{\sqrt{\sparsityParam \nsize}}
    \le
    C \nclusters^{1/2} {\frac{\sqrt{\log \nsize}}{\sparsityParam \nsize}}.
  \end{EQA}
  The required bounds follow from the following inequalities
  \begin{EQA}[c]
    \frac{\bigl\|\communityMatrixEstimate - \mathrm{\Pi}_{\basisMatrix} \communityMatrix \mathrm{\Pi}_{\basisMatrix}^{\T}\bigr\|_{F}}{\|\communityMatrix\|_{F}}
    \le
    C \nclusters^{1/2} \condNumber^{2}(\nodeCommunityMatrix) \lambda_{max}(\nodeCommunityMatrix) \errorAdjacency + C \nclusters^{1/2} \frac{\condNumber(\probMatrix)}{\lambda_{max}(\communityMatrix)} \frac{\|\adjacencyMatrix - \probMatrix\|}{\lambda_{\nclusters}^2(\nodeCommunityMatrix)}
    \le
    C \nclusters \sqrt{\frac{\log \nsize}{\sparsityParam^2 \nsize}}
  \end{EQA}
  and
  \begin{EQA}[c]
    \frac{\bigl\|\nodeCommunityMatrixEstimate - \nodeCommunityMatrix \mathrm{\Pi}_{\basisMatrix}^{\T}\bigr\|_{F}}{\|\nodeCommunityMatrix\|_{F}}
    \le
    C \nclusters^{1/2} \condNumber^{3}(\nodeCommunityMatrix) \lambda_{max}(\nodeCommunityMatrix) \, \errorAdjacency + C \nclusters^{1/2} \condNumber(\nodeCommunityMatrix) \frac{\|\adjacencyMatrix - \probMatrix\|}{\lambda_{\nclusters}(\probMatrix)}
    \le
    C \nclusters \sqrt{\frac{\log \nsize}{\sparsityParam^2 \nsize}},
  \end{EQA}
  which hold with probability at least \(1 - \nsize^{-r}\) for the properly chosen constant \(C\).

\subsection{Proof of Lemma~\ref{lemma: probMaxValue}}
  We start by noting that
  \begin{EQA}[c]
  \label{eq: maxEstimate}
    \max_{i,j} \probMatrix_{i, j} = \max_{i, j}  \nodeWeights_{i} \communityMatrix \nodeWeights_{j}^{\T}.
  \end{EQA}
  As we assume, that there exist pure nodes for each community then we can take community membership vectors that correspond to pure nodes for the communities, which have maximum inter-community probability. Due to the fact, that all \(\nodeWeights_{i}\) are convex combinations such a choice of nodes will give maximum to~\eqref{eq: maxEstimate}. Thus, we obtain \(\max_{i,j} \probMatrix_{i, j} = \max_{k, l} \communityMatrix_{k, l} = \sparsityParam\).

\subsection{Proof of Lemma~\ref{lemma: eigenBoundProb}}
  Let us consider the behaviour of \(k\)-th eigenvalue of matrix \(\probMatrix\) for \(k = 1, \dots, \nclusters\):
  \begin{EQA}[c]
    \lambda_{k}(\probMatrix) = \lambda_{k} \bigl( \nodeCommunityMatrix \communityMatrix \nodeCommunityMatrix^{\T}\bigr) = \lambda_{k} \bigl( \nodeCommunityMatrix \basisMatrix \probEigenvalues \basisMatrix^{\T} \nodeCommunityMatrix^{\T}\bigr) = \lambda_{k} \bigl( \basisMatrix^{\T} \nodeCommunityMatrix^{\T} \nodeCommunityMatrix \basisMatrix \probEigenvalues\bigr).
  \end{EQA}
  Let us consider matrix
  \begin{EQA}[c]
    \eigenMatrix = \nodeCommunityMatrix^{\T} \nodeCommunityMatrix  = \sum_{i = 1}^{\nsize}\nodeWeights_{i}^{\T} \nodeWeights_{i}.
  \end{EQA}
  The expectation of matrix \(\eigenMatrix\) is given by the following formula:
  \begin{EQA}[c]
    \EE \eigenMatrix = \nsize \EE \bigl[\nodeWeights_{1}^{\T} \nodeWeights_{1}\bigr].
  \end{EQA}
  We note that if distribution of \(\nodeWeights_{1}\) has a non-zero mass at all ``pure'' nodes, then the matrix \(\EE \eigenMatrix\) is positive definite. Consequently, we can state that
  \begin{EQA}[c]
    \lambda_{min}(\EE \eigenMatrix) = \Theta(\nsize) ~ \text{and} ~ \lambda_{max}(\EE \eigenMatrix) = \Theta(\nsize).
  \end{EQA}
  We proceed by bounding the fluctuations of eigenvalues of matrix \(\eigenMatrix\) around the mean with help of following lemma:
  \begin{lemma}
  \label{lemma: eigenConcentration}
    There exist such constants \(c\) and \(C\) depending only on distribution of vector \(\nodeWeights_{i}\) such that
    \begin{EQA}
      && \PP\Bigl\{\lambda_{min}\Bigl(\sum_{i = 1}^{\nsize} \nodeWeights_{i}^{\T} \nodeWeights_{i}\Bigr) \leq c \nsize\Bigr\} \leq \nclusters \ex^{-c \nsize / 4};
      \\
      && \PP\Bigl\{\lambda_{max}\Bigl(\sum_{i = 1}^{\nsize} \nodeWeights_{i}^{\T} \nodeWeights_{i}\Bigr) \geq C \nsize\Bigr\} \leq \frac{\nclusters}{2^{C \nsize}}.
    \end{EQA}    
  \end{lemma}
  \begin{proof}
    Let us note that the every matrix \(\nodeWeights_{i}^{\T} \nodeWeights_{i}\) is positive semidefinite and 
    \begin{EQA}[c]
      \lambda_{max}(\nodeWeights_{i}^{\T} \nodeWeights_{i}) \leq 1.
    \end{EQA}
    Let us take \(t = 0.5\) and note that \(\lambda_{min}\Bigl(\EE \sum_{i = 1}^{\nsize} \nodeWeights_{i}^{\T} \nodeWeights_{i}\Bigr) = c \nsize\). Then by matrix Chernoff bound (see Theorem~\ref{theorem: chernoff} and Corollary~\ref{corollary: chernoffSimple}) we obtain
    \begin{EQA}
      && \PP\Bigl\{\lambda_{min}\Bigl(\sum_{i = 1}^{\nsize} \nodeWeights_{i}^{\T} \nodeWeights_{i}\Bigr) \leq  \frac{c}{2}\nsize\Bigr\} \leq \nclusters \ex^{-c \nsize / 8}.
    \end{EQA}
    We further note that \(\lambda_{max}\Bigl(\EE \sum_{i = 1}^{\nsize} \nodeWeights_{i}^{\T} \nodeWeights_{i}\Bigr) = C \nsize\) and again we can bound
    \begin{EQA}
      && \PP\Bigl\{\lambda_{max}\Bigl(\sum_{i = 1}^{\nsize} \nodeWeights_{i}^{\T} \nodeWeights_{i}\Bigr) \geq 2 C \ex \nsize\Bigr\} \leq \nclusters \frac{1}{2^{2 C \ex \nsize}}.
    \end{EQA}
    This completes the proof of the desired result.
  \end{proof}
  The result of Lemma~\ref{lemma: eigenConcentration} directly implies bounds~\eqref{eq: eigenBoundNodeCommunity}. Further by Lemma~\ref{lemma: eigenConcentration} there exist such constants \(c\) and \(C\) that with probability at least \(1 - \nclusters (\ex^{-c \nsize / 4} + \frac{\nclusters}{2^{C \nsize}})\) it holds
  \begin{EQA}[c]
    \lambda_{min}(\eigenMatrix) \geq c \nsize ~ \text{and} ~ \lambda_{max}(\eigenMatrix)  \leq C \nsize.
  \end{EQA}
  Finally we get
  \begin{EQA}[c]
    \lambda_{min}(\probMatrix) \geq \lambda_{min}\bigl(\communityMatrix\bigr) \cdot \lambda_{min}\bigl(\nodeCommunityMatrix^{\T} \nodeCommunityMatrix\bigr) \geq c \lambda_{min}(\communityMatrix) \nsize
  \end{EQA}
  and
  \begin{EQA}[c]
    \lambda_{max}(\probMatrix) \leq \lambda_{max}\bigl(\communityMatrix\bigr) \cdot \lambda_{max} \bigl(\nodeCommunityMatrix^{\T} \nodeCommunityMatrix\bigr) \leq C \lambda_{max}(\communityMatrix) \nsize.
  \end{EQA}
  The respective bounds from below for \(\lambda_{max}(\probMatrix)\) and from above for \(\lambda_{min}(\probMatrix)\) come from the fact that \(c \nsize \identity \prec \eigenMatrix \prec C \nsize \identity\).

\end{document}